\documentclass[twoside,12pt]{article}

\usepackage[utf8]{inputenc}
\usepackage{amssymb,amsmath,amsthm,subfiles,mathtools,mathrsfs,comment,dsfont,mathrsfs,graphicx,enumerate,caption,wrapfig,float,hyperref,titlesec,caption,subcaption,afterpage}
\usepackage[dvipsnames]{xcolor}
\usepackage[T1]{fontenc}
\usepackage{titleps,lipsum}

\makeatletter
\newcommand*\bigcdot{\mathpalette\bigcdot@{1}}
\newcommand*\bigcdot@[2]{\mathbin{\vcenter{\hbox{\scalebox{#2}{$\m@th#1\bullet$}}}}}
\makeatother

\makeatletter
\newcommand*{\rom}[1]{\expandafter\@slowromancap\romannumeral #1@}
\makeatother

\numberwithin{equation}{section}

\author{
 \scshape Lian Haeming \\
  \textit{Queen Mary University of London}
}

\graphicspath{{\string~/Desktop/Tex Files/Lyapunov rectangle boxes.jpg}}
    
\hypersetup{
    colorlinks=true,
    linkcolor=MidnightBlue,
    citecolor = MidnightBlue,
    linktoc = all
}

\usepackage{xcolor}
\pagecolor{white}
\usepackage[margin=0.72in]{geometry}
\title{\scshape \bfseries \Large quasiballistic transport for discrete  one-dimensional quasiperiodic  schr\"odinger operators }
\date{}
\newcommand\restr[2]{{
  \left.\kern-\nulldelimiterspace 
  #1 
  \vphantom{\big|} 
  \right|_{#2} 
  }}

\newcommand\R{\mathbb{R}}
\newcommand\C{\mathbb{C}}
\newcommand\Z{\mathbb{Z}}
\newcommand\N{\mathbb{N}}
\newcommand\Q{\mathbb{Q}}

\newcommand\T{\mathbb{T}}

\makeatletter
\newcommand{\newreptheorem}[2]{\newtheorem*{rep@#1}{\rep@title}\newenvironment{rep#1}[1]{\def\rep@title{#2 \ref*{##1}}\begin{rep@#1}}{\end{rep@#1}}}
\makeatother

\def\restrict#1{\raise-.5ex\hbox{\ensuremath|}_{#1}}

\newtheorem{thm}{\normalfont\bfseries Theorem}
\newtheorem{lemma}[thm]{\normalfont\bfseries  Lemma}
\newtheorem{claim}[thm]{\normalfont\bfseries Claim}
\newtheorem{rmk}[thm]{\normalfont\bfseries Remark}
\newtheorem{prop}[thm]{\normalfont\bfseries Proposition}

\newreptheorem{theorem}{\normalfont\bfseries Theorem}
\newreptheorem{lemma}{\normalfont\bfseries Lemma}
\newreptheorem{prop}{\normalfont\bfseries Proposition}
\newreptheorem{cor}{\normalfont\bfseries Corollary}

\makeatletter
\let\hdrtitle\@title
\newpagestyle{main}{%
  \sethead[][\scshape{Quasiballistic motion}][\thepage] 
  {}{\scshape{L. Haeming}}{\thepage}} 
\makeatother
\titlespacing*{\section}{0pt}{4ex}{1.5ex}

\titleformat{\section}[block]{\color{black}\scshape\filcenter}{\thesection.}{0.5em}{}

\begin{document}
\renewcommand{\abstractname}{\vspace{-\baselineskip}}
\maketitle

\begin{abstract}
We obtain (up to logarithmic scaling) the power-law lower bound $M_{p}(T_{k})\gtrsim T_{k}^{(1-\delta)p}$ on a subsequence $T_{k}\rightarrow\infty$,  uniformly across $p>0$, for discrete one-dimensional quasiperiodic Schr\"odinger operators with frequencies satisfying $\beta(\alpha)>\frac{3}{\delta}\min_{\sigma}\gamma$. We achieve this by obtaining a quantitative ballistic lower bound for the Abel-averaged time evolution of general periodic Schr\"odinger operators in terms of the bandwidths. A similar result without uniformity, which assumes $\beta(\alpha)>\frac{C}{\delta}\min_{\sigma}\gamma$, was obtained earlier by Jitomirskaya and Zhang, for an implicit constant $C<\infty$. 
\end{abstract}

\section*{Introduction}
Let $\T = \R/\Z$ be the unit circle. To each phase $\theta\in \T$, frequency $\alpha\in\R$ and Lipschitz sampling function $f:\T\rightarrow\R$ we associate a discrete  Schr\"odinger operator $H_{\alpha,\theta}:\ell^2(\mathbb{Z}) \rightarrow \ell^2(\mathbb{Z})$, where
\begin{equation}\label{dyn:01}
(H_{\alpha,\theta}\psi)(n)=\psi(n-1)+\psi(n+1)+f(\theta+n\alpha)\psi(n).
\end{equation}

We are interested in the rate of spreading of the solution $\psi_{t}=e^{-itH_{\alpha,\theta}}\delta_{0}$ of the Schr\"odinger equation $\partial_{t}\psi_{t}=-iH_{\alpha,\theta}\psi_{t}$ with initial condition given by the canonical vector $\psi_{0}=\delta_{0}=(\dots,0,1,0,\dots)$.  One way to quantify the rate of spreading is through the Abel-averaged moments of the position operator, 
\begin{equation}\label{dyn:02}
M_{\alpha,\theta,p}(T) =\frac{2}{T}\int_{0}^{\infty}e^{-\frac{2t}{T}}\sum_{n\in\Z} |n|^{p}|\langle \delta_{n}, e^{-itH_{\alpha,\theta}}\delta_{0}\rangle|^{2}\,dt
\end{equation} for $p>0$. The ballistic bound (see \eqref{dyn:ballistic}) implies $M_{\alpha,\theta,p}(T)\leq C(T^{p}+1)$ for any $T>0$ and $p>0$.  Namely,  the average distance from the particle to the origin $\approx M_{\alpha,\theta,p}^{1/p}$ grows no faster than linearly in time.  

Singular continuous spectra is encountered frequently in the quasiperiodic setting, even for basic models such as the almost Mathieu operator (AMO) with cosine sampling  $f(\theta)=2\lambda\cos(2\pi\theta)$, $\lambda>0$. The direct consequences of singular continuous spectra on the dynamics is not well understood, but such models can manifest surprising dynamical behaviour such as being almost ballistic on some time-scales while almost localised on others.  

The relationship between the arithmetic properties of the frequency (i.e.\ how well approximable it is by rationals) and the Lyapunov exponent \eqref{dyn:lyapunov} determines where the spectral measure is continuous. In the case of Liouville frequencies, Jitomirskaya \cite{Ji94} introduced the rate of exponential growth of the denominators of the canonical continued fraction approximants ${p_{m}}/{q_{m}}$  of the frequency $\alpha\in\R\setminus\Q$,
$$
\beta(\alpha)=\limsup_{m\rightarrow\infty} \frac{\log q_{m+1}}{q_{m}}.
$$
It is known (see below) that the spectrum is continuous (and hence singular continuous if the Lyapunov exponent is positive  $\min_{\R}\gamma>0$) in the region where $\beta(\alpha)$ is greater than the Lyapunov exponent
\begin{equation}\label{dyn:lyapunov}
\gamma(E) = \lim_{n\rightarrow\infty}  \frac{1}{n}\int_{\T}\log\|\Phi_{\alpha,\theta,[0,n-1]}(E)\|\,d\theta
\end{equation} for $\alpha\in\R\setminus\Q$, where 
$\Phi_{\alpha,\theta,[0,n-1]}(E)\in\text{SL}_{2}(\R)$ is the $n$-step transfer matrix  \eqref{dyn:36} associated with the operator $H_{\alpha,\theta}$. See for example the lecture notes of Viana \cite{Vi14} for the existence of the limit \eqref{dyn:lyapunov}.

The usual approach to obtain dynamical lower bounds is by obtaining suitable continuity properties of the spectral measure. One example is the Guarneri -- Combes -- Last bound \cite[Thm 6.1]{La96} which shows (roughly speaking) that the moments grow polynomially with order given by the upper Hausdorff dimension of the spectral measure. The construction of Last \cite[Thm 7.2]{La96}  shows that the supercritical ($\lambda>1$) AMO with certain extremely Liouville  frequencies (i.e.\ with $\beta(\alpha)=\infty$) has purely zero Hausdorff dimensional spectrum yet the transport is almost ballistic on some time-scales.   Zero Hausdorff dimensional spectrum is a phenomena not limited to the supercritical AMO. A theorem of Simon \cite{Si07} states that the support of the spectral measure of an ergodic Schr\"odinger operator with positive Lyapunov exponent \eqref{dyn:lyapunov} is of zero logarithmic capacity and therefore of zero Hausdorff dimension. The dynamical behaviour associated with quasiperiodic operators with positive Lyapunov exponent therefore require a more nuanced description.

In the regime of positive Lyapunov exponent, without any assumptions on  the arithmetic properties of the frequency $\alpha\in\R\setminus\Q$, the moments obey a sub-power-law bound on an unbounded \emph{subsequence} of times. If however, the frequency satisfies $\beta(\alpha)=0$ then the moments obey a sub-power-law bound at \emph{all} times. Both of these facts were established for trigonometric polynomials by Damanik and Tcheremchantsev \cite{DaTc07}. Combined with Last's example, these results show that the transport of the AMO is essentially localised (at all times) for frequencies satisfying $\beta(\alpha)=0$, while for other frequencies with $\beta(\alpha)=\infty$, its transport is essentially ballistic on some time-scales (but simultaneously, essentially localised on others).  This suggests a relationship between the size of $\beta(\alpha)$ and the growth of the moments on such time-scales.

\begin{thm}\label{dyn:theorem}
Let $H_{\alpha,\theta}$ be as in \eqref{dyn:01} with associated continuous Lyapunov exponent $\gamma$. Let $\sigma_{\alpha}$ denote the deterministic spectrum of $H_{\alpha,\theta}$. There exists $c>0$ such that for any $0<\delta<\frac{1}{2}$, if $\beta(\alpha)>\frac{3}{\delta}\min_{\sigma_{\alpha}}\gamma$ then there exists a sequence $T_{k}\rightarrow\infty$ such that for each $k\geq1$ we have 
\begin{equation}\label{dyn:02.5}
\min_{\theta\in\T} M_{\alpha,\theta,p}(T_{k})>\frac{cT_{k}^{(1-\delta)p}}{\log^{10} T_{k}}
\end{equation} for every $p>0$. 
\end{thm}

The relationship between the size of the exponent $\beta(\alpha)$ and the upper transport exponent has been studied earlier by Jitomirskaya and Zhang \cite[Thm 7]{JiZh21} who showed that for an implicit constant $C<\infty$; for any $0<\delta<1$; if $\beta(\alpha)>\frac{C}{\delta}\min_{\sigma_{\alpha}}\gamma$, then the packing dimension of the spectral measure is at least $1-\delta$ and therefore, so is the upper transport exponent, for each $\theta\in\T$ and $p>0$, since the upper transport exponent is bounded from below by the packing dimension of the spectral measure by the work of Guarneri and Schulz-Baldes \cite{GuSB02}. 

Lower bounds on the upper  transport exponent for models with singular continuous spectrum does not provide a significant physical distinction against models with pure point spectrum in light of the example of del Rio -- Jitomirskaya -- Last -- Simon \cite{dRJiLaSi96} who constructed an operator with upper transport exponent equal to $1$, despite having pure point spectrum. Our lower bound therefore offers a stronger characterisation of particle behaviour than that of \cite{JiZh21}, as it is a lower bound on the moments themselves, uniformly across all positive values of $p$. We require a weaker assumption on $\beta(\alpha)$, but the lower bound of \cite{JiZh21} permits larger values of $\delta$. Our approach, as well as that of \cite{JiZh21} require continuity of the Lyapunov exponent at a  minimum on the spectrum. 
Bourgain -- Jitomirskaya \cite[Thm 1]{BoJi02} show that the Lyapunov exponent associated with quasiperiodic Schr\"odinger operators with real analytic $f$  is jointly continuous in the energy $E\in\R$ and irrational frequency $\alpha\in\R\setminus\Q$. 

Last's example uses the fact that the transport associated with periodic operators is ballistic and then shows that the limit captures ballisticity on some time-scales, since the frequency is Liouville.  We draw inspiration from Last's construction in that we obtain a quantitative ballistic lower bound on the entries of the time evolution of general periodic operators (see Lemma \ref{dyn:periodiclower}), in terms of the bandwidths. We then extend this to the limit on a subsequence of times. The theorem then follows from Proposition \ref{dyn:bandwidths} which lower-bounds the bandwidths in terms of the Lyapunov exponent of the limiting operator.

Absence of pure point spectrum for extremely Liouville frequencies $\beta(\alpha)=\infty$ was established by Gordon \cite{Go76}. Avron and Simon \cite{AvSi82} used Gordon's theorem to show that the supercritical AMO has singular continuous spectrum. In the regime of positive Lyapunov exponent, Kotani \cite{Ko84} and Gordon \cite{Go76} imply singular continuous spectrum for extremely Liouville frequencies. By repeating the arguments of Gordon in the usual way for finite $\beta(\alpha)$, one obtains that the spectral measure is continuous on the set $\{E:\beta(\alpha)>2\gamma(E)\}$ -- the factor of $2$ arising from the fact that one has to approximate the solution along double periods, see the proof of Lemma \ref{dyn:nonatomic}. Avila -- You -- Zhou \cite{AvYoZh15} showed that for $0<\beta(\alpha)<\infty$, the spectrum of the AMO is purely singular continuous for all $\theta\in\T$ if $1\leq|\lambda|<e^{\beta(\alpha)}$ and pure point with exponentially decaying eigenfunctions for a.e $\theta\in\T$, if $|\lambda|>e^{\beta(\alpha)}$.  Jitomirskaya -- Yang \cite{JiYa17} have ruled-out pure point spectrum in the region $\{E:\beta(\alpha)>\gamma(E)\}$, for all meromorphic sampling functions $f$. See also Yang -- Zhang \cite{YaZh19} for an extension of this fact to a larger class of sampling functions. See Jitomirskaya -- Liu \cite{JiLi17} for a complete description of the spectrum of the Maryland model, in terms of the arithmetic properties of the phase and frequency. 

Various upper bounds on the transport have been established since the work of  Damanik -- Tcheremchantsev \cite{DaTc07}. Jitomirskaya -- Mavi \cite{JiMa16} extended their result to piecewise H\"older sampling functions, and subsequently, Han -- Jitomirskaya \cite{HaJi18} extended it to a wider class of ergodic potentials in the multi-frequency setting. These results are limited to transport exponents. Jitomirskaya -- Powell \cite{JiPo22} derived power-logarithmic upper bounds on the transport for any fixed $\theta\in\T$, which was later improved by Jitomirskaya -- Liu  \cite{JiLi21} to long-range operators, and then by Shamis -- Sodin \cite{ShSo23}, followed by Liu \cite{Li23}, to long-range operators in arbitrary dimensions, uniformly across phase $\theta\in\T$.

\section*{Acknowledgements}
It is a pleasure to thank Mira Shamis for proposing this work and Sasha Sodin for his support towards its completion. The author is grateful to the Weizmann Institute of Science for their hospitality during the completion of this work. This work was supported by the EPSRC PhD fellowship and supported in part by an EPSRC research grant (EP/X018814/1) and by a Philip Leverhulme Prize of the Leverhulme Trust (PLP-2020-064).

\section{Time evolution of periodic operators}

Our general strategy for the proof of the theorem is the following. The main ingredient is Lemma \ref{dyn:periodiclower}, which provides a lower bound on the Abel-average \eqref{dyn:10} of the sum of two entries of the time evolution operator $e^{-itH_{q}}$ associated with a general periodic operator
$$
H_{q}:\ell^{2}(\Z)\rightarrow\ell^{2}(\Z), \quad H_{q}=\Delta+V_{q}
$$
where $V_{q}$ is a periodic sequence of period $q\geq1$. Lemma \ref{dyn:03.5} provides the explicit expression describing  the averaged entries in terms of the resolvent operator associated with $H_{q}$, which is to be lower bounded in Lemma \ref{dyn:periodiclower}. The expression \eqref{dyn:11} contains as a factor the canonical  spectral measure \eqref{dyn:12}  associated with the Floquet matrix. The only assumption of Lemma \ref{dyn:periodiclower} is therefore a uniform (over the Floquet number $\varkappa\in[0,\pi/q]$) lower bound on the spectral measure evaluated at an interval. In Section \ref{dyn:secuniflowerbound} we show that this assumption holds (also uniformly in phase and period) for the Floquet matrix associated with the periodic operator $H_{\alpha_{m},\theta}$ for $\alpha_{m}\rightarrow\alpha$ where $\alpha\in\R\setminus\Q$ satisfies the assumptions in the theorem. The proof of the theorem mainly involves showing that a similar lower bound to Lemma \ref{dyn:periodiclower} also holds for the limiting quasiperiodic operator.

It is well known (see the proof of Lemma \ref{dyn:03.5}) that periodic operators $H_{q}$ are unitarily equivalent to a multiplication operator $M_{q} : L^{2}(\T_{q}\mapsto\C^{q})\rightarrow L^{2}(\T_{q}\mapsto\C^{q})$, $\T_{q}= \R/\frac{2\pi}{q}\Z$, which acts as a multiplication by the matrix-valued function 
\begin{equation}\label{dyn:floquet}
A_{q}:\T_{q}\rightarrow\C^{q\times q},\quad A_{q}(\varkappa) = 
\begin{pmatrix}
V_{q}(-\frac{q}{2}) & 1 & & e^{iq\varkappa} \\ 
1 & \ddots &\ddots &  \\
 & \ddots & \ddots & 1 \\
e^{-iq\varkappa}&  & 1 & V_{q}(\frac{q}{2}) \\
\end{pmatrix} 
\end{equation} known as the Floquet matrix.

Let $\lambda_{j}(\varkappa)\in\R$ denote the $j$-th  eigenvalue of $A_{q}(\varkappa)$, counting from the left $\lambda_{j}(\varkappa)\leq \lambda_{j+1}(\varkappa)$. By unitary equivalence, the spectrum of the periodic operator $H_{q}$ is given by the spectrum of the multiplication operator $M_{q}$ which itself is the union of the spectrum of the Floquet matrix over all $\varkappa\in[0,\pi/q]$. The spectrum of $H_{q}$ is the union of $q$ closed intervals (which are called bands) $B_{q}^{(j)}$, 
$$
\sigma(H_{q})=\bigcup_{j=1}^{q}B_{q}^{(j)}, \quad B_{q}^{(j)}=\bigcup_{\varkappa\in [0,\pi/q]}\{\lambda_{j}(\varkappa)\}
$$
with mutually disjoint interiors. By characteristic polynomial \eqref{dyn:06} considerations, the eigenvalues $\lambda_{j}$ are easily seen to be monotonic as functions of $\varkappa\in[0,\pi/q]$, whose derivatives alternate in sign according to parity of $1\leq j\leq q$.  Our lower bound of Lemma \ref{dyn:periodiclower} is given in terms of the bandwidths
 $$
 \ell_{j}=\ell_{j,q}=|B_{q}^{(j)}|=|\lambda_{j}(\pi/q)-\lambda_{j}(0)|.
 $$ 

Most of the effort in this paper goes into estimating the RHS of \eqref{dyn:11},  which is an explicit expression for the Abel-averaged entries \eqref{dyn:10} in terms of the eigen-pairs $(\lambda_{j}(\varkappa),\Psi_{\varkappa}^{(j)})$ of the Floquet matrix \eqref{dyn:floquet}, 
\begin{equation}\label{dyn:10}
P_{q,T}(n)=\frac{2}{T}\int_{0}^{\infty}(|\langle \delta_{n}, e^{-itH_{q}}\delta_{0}\rangle|^{2}+|\langle \delta_{n+1}, e^{-itH_{q}}\delta_{1}\rangle|^{2})e^{-\frac{2t}{T}}\,dt.
\end{equation}

Note that the entries \eqref{dyn:10} differ from \eqref{dyn:02}. The theorem is actually proved for the summed entries \eqref{dyn:10} as opposed to as stated in \eqref{dyn:02}. Summing the two entries $(0,nq)$ and $(1,nq+1)$ allows us to express the quantum probability $P_{q,T}(nq)$ in terms of the canonical spectral measure \eqref{dyn:12}, whose support coincides with the spectrum.

\begin{lemma}\label{dyn:03.5}
Let $H_{q}$ be a periodic Schr\"odinger operator of period $q\geq1$. Let $(\lambda_{j}(\varkappa),\Psi_{\varkappa}^{(j)})$ denote the $j$-th eigenpair of the associated Floquet matrix \eqref{dyn:floquet}. Let $P_{q,T}$ be the quantity defined in \eqref{dyn:10}. For any $n\in\Z$,
\begin{equation}\label{dyn:11}
P_{q,T}(nq)=\frac{1}{\pi T}\int_{\R}\bigg|\sum_{j=1}^{q}\int_{0}^{\frac{\pi}{q}} \frac{\cos(nq\varkappa)\varphi_{j}(\varkappa)}{\lambda_{j}(\varkappa)-E-iT^{-1}} \frac{d\varkappa}{\frac{\pi}{q}}\bigg|^{2}\,dE
\end{equation}
where $\varphi_{j}(\varkappa)=|\langle \Psi_{\varkappa}^{(j)},e_{0}\rangle|^{2}+|\langle \Psi_{\varkappa}^{(j)},e_{1}\rangle|^{2}$ and $e_{k}\in\C^{q}$ is the $k$-th canonical basis vector for $\C^{q}$.
\end{lemma}

Lemma \ref{dyn:03.5} is obtained in the usual way (see the end of this section) by expressing the probabilities $P_{q,T}(nq)$ in terms of the corresponding entries of the resolvent operator associated with $H_{q}$, followed by diagonalizing the periodic operator $H_{q}$ in the Fourier space and then changing to the eigenbasis of the Floquet matrix.

In general there are no suitable lower bounds on each individual function $\varphi_{j}$ other than the usual exponential lower bound. For our purposes, an exponential lower bound on $\varphi_{j}$ does not suffice. One way around this, however, is that the functions $\{\varphi_{j}(\varkappa)\}_{j}$ do define the canonical  spectral measure 
\begin{equation}\label{dyn:12}
\mu_{\varkappa,q}= \sum_{j=1}^{q} \varphi_{j}(\varkappa)\delta_{\lambda_{j}(\varkappa)}
\end{equation}
 associated with the Floquet matrix $A_{q}(\varkappa)$, where $\delta_{\lambda_{j}(\varkappa)}$ is the Dirac measure at the eigenvalue $\lambda_{j}(\varkappa)$.

The only assumption of Lemma \ref{dyn:periodiclower} is for the measures $\mu_{\varkappa,q}$ evaluated at an interval $I\subset\R$ to be uniformly (in $\varkappa\in[0,\pi/q]$) bounded from below by a positive number $\eta>0$. This is enough (see Claim \ref{dyn:chebyshev}, below) to deduce that there exists a $1\leq j\leq q$ such that $B_{q}^{(j)}\cap I\neq\varnothing$ and
$$
|\{\varkappa:\varphi_{j}(\varkappa)>\eta/q\}|>\frac{\pi}{2q^{2}}.
$$ In the proof of the theorem, we shall choose the interval $I=B_{\varepsilon}(E_{0})\subset\R$ to be the vicinity of the minimum of the Lyapunov exponent, where $B_{\varepsilon}(E_{0})$ denotes a ball of radius $\varepsilon>0$ centred at the point in the spectrum $E_{0}\in\sigma_{\alpha}$ where $\gamma(E_{0})=\min_{\sigma_{\alpha}}\gamma$.

\begin{lemma}\label{dyn:periodiclower} Let $H_{q}$ be a periodic Schr\"odinger operator of period $q\geq1$ and let $B_{q}^{(j)}$ denote the $j$-th band in its spectrum with width $\ell_{j}$. Let $\mu_{\varkappa,q}$ be the canonical  spectral measure \eqref{dyn:12} and $P_{q,T}$ the probabilities  \eqref{dyn:10}. Let $I\subset\R$ be an  interval and suppose $\inf_{\varkappa\in[0,\pi/q]}\mu_{\varkappa,q}(I)>\eta>0$. There exist a band $B_{q}^{(j)}\cap I\neq\varnothing$ and constants $0<c,c_{1},C<\infty$ such that for $T>\frac{C}{c_{1}}\eta^{-2}q^{8}\ell_{j}^{-2}+1$, 
\begin{equation}\label{dyn:13}
P_{q,T}(nq) > \frac{c\eta^{2}}{q^{6}\ell_{j}T}
\end{equation} 
provided  $C\eta^{-1}q^{4}\ell_{j}^{-1}<n<c_{1}\eta q^{-4}\ell_{j}T$.
\end{lemma}

\begin{proof}
The imaginary part of the function inside the modulus in \eqref{dyn:11} is given by the function $\frac{q}{\pi T}h$, where 
\begin{equation}
h(E)=\sum_{i=1}^{q}\int_{0}^{\frac{\pi}{q}}g_{i}(\varkappa)\,d\varkappa,\quad g_{i}(\varkappa) = \frac{\cos(nq\varkappa)\varphi_{i}(\varkappa)}{(\lambda_{i}(\varkappa)-E )^{2}+T^{-2}}.
\end{equation}  since the squared modulus is at least the squared imaginary part $|\cdot|^{2}\geq \Im^{2}$,  we have 
\begin{equation}\label{dyn:14}
P_{q,T}(nq) \geq \frac{q^{2}}{\pi^{3}T^{3}}\int_{\R}(h(E))^{2}\,dE
\end{equation} having factored out the coefficient $\frac{q}{\pi T}$ of the imaginary part.  Now, for an indexing set $K\subset \N$, for a certain eigenvalue $\lambda_{j}$ and for a subinterval $\widetilde{I}_{k}\subset[0,\pi/q]$, all of which we shall define in the next paragraph, the problem is reduced to
\begin{equation}\label{dyn:55}
\int_{\R}(h(E))^{2}\,dE > \sum_{k\in K}\int_{\lambda_{j}(\widetilde{I}_{k})} (h(E))^{2}\,dE \geq \#K\min_{k\in K}|\lambda_{j}(\widetilde{I}_{k})|\min_{E\in\lambda_{j}(\widetilde{I}_{k})}(h(E))^{2}
\end{equation} where the set $\lambda_{j}(\widetilde{I}_{k})=\{\lambda_{j}(\varkappa):\varkappa\in\widetilde{I}_{k}\}\subset B_{q}^{(j)}$ is the image of the interval $\widetilde{I}_{k}$ under $\lambda_{j}$. 

The eigenvalue $\lambda_{j}$ in \eqref{dyn:55} is chosen specifically to be the one given by Claim \ref{dyn:chebyshev}, below, for which $B_{q}^{(j)}\cap I\neq \varnothing$ and 
\begin{equation}\label{dyn:71}
|\{\varkappa:\varphi_{j}(\varkappa)>\eta/q\}|>\frac{\pi}{2q^{2}}.
\end{equation}
 The subinterval  $\widetilde{I}_{k}=[\frac{k\pi}{nq}-\frac{3\pi}{8nq},\frac{k\pi}{nq}+\frac{3\pi}{8nq}]$ is chosen so as to lower bound the cosine by $\min_{\varkappa\in\widetilde{I}_{k}}|\cos(nq\varkappa)|>\frac{1}{\pi}$. The indexing set $K\subset\N$ is the set of indices $k\in K$ for which the interior  $\text{int}(I_{k})$ intersects with the set $\{\varkappa:\varphi_{j}(\varkappa)>\eta/q\}$, where the subinterval $I_{k}=[\frac{k\pi}{nq}-\frac{\pi}{2nq},\frac{k\pi}{nq}+\frac{\pi}{2nq}]$ contains $\widetilde{I}_{k}$ which partitions the half-torus $[0,\pi/q]$ by the roots of the cosine: $\cos(nq(\frac{k\pi}{nq}\pm\frac{\pi}{2nq}))=0$. The idea is to place the location $\lambda_{j}^{-1}(E)$ of the main peak of the $j$-th function $g_{j}$ inside the interval $\widetilde{I}_{k}$ where we have a lower bound on both the cosine and on the function $\varphi_{j}$.

The quantities $\#K$ and $|\lambda_{j}(\widetilde{I}_{k})|$ can easily be estimated using the lower bounds \eqref{dyn:71} and \eqref{dyn:07.5}, respectively. Therefore the bulk of the problem lies in estimating the factor $\min_{E\in\lambda_{j}(\widetilde{I}_{k})}(h(E))^{2}$ in \eqref{dyn:55}, for which the main idea is to split $h(E)$ as in \eqref{dyn:53} and estimate the three terms separately. Note that the centres of the subintervals $I_{k},\widetilde{I}_{k}$ were chosen so that $\cos(nq\frac{k\pi}{nq})=(-1)^{k}$, so the cosine $\cos(nq\varkappa)$ is positive on the even subintervals $\varkappa\in\widetilde{I}_{2k}$, and negative on the odd ones. Since $h$ is being squared, we could either place $E\in\lambda_{j}(\widetilde{I}_{2k})$ and bound $h(E)$ from below or place $E\in\lambda_{j}(\widetilde{I}_{2k+1})$ and bound $h(E)$ from above. The two procedures are similar, so we only consider the former.  We shall assume from now on that the energy $E\in\lambda_{j}(\widetilde{I}_{2k})$ is fixed, for some $2k\in K$. Writing 
\begin{equation}\label{dyn:53}
h(E) =\int_{\widetilde{I}_{2k}}g_{j}(\varkappa)\,d\varkappa+\int_{\widetilde{I}_{2k}^{\mathsf{c}}}g_{j}(\varkappa)\,d\varkappa + \sum_{i\neq j} \int_{0}^{\frac{\pi}{q}} g_{i}(\varkappa)\,d\varkappa, 
\end{equation}
our aim is to show that the (positive) area obtained from the first term  outweighs the negative area obtained from the second two terms. 

Let us start bounding the first term in \eqref{dyn:53}, from below.  The first task is to bound $g_{j}(\varkappa)$ from below for each $\varkappa\in\widetilde{I}_{2k}$. For each $\varkappa\in\widetilde{I}_{2k}$, the denominator of $g_{j}(\varkappa)$ will be bounded from above by 
\begin{equation}\label{dyn:57}
|\lambda_{j}(\varkappa)-E|\leq \max_{\varkappa'\in\widetilde{I}_{2k}}|\lambda'_{j}(\varkappa')||\varkappa-\lambda_{j}^{-1}(E)|
\end{equation} whereas the numerator will be bounded from below by the constant $\frac{1}{\pi}\min_{\varkappa\in\widetilde{I}_{2k}}\varphi_{j}(\varkappa)$.

According to  \eqref{dyn:56}, below, the estimates on the derivative of the eigenvalue $\lambda_{j}$ deteriorates near the edges of the half-torus, so we deal with this issue by removing the edges of the half-torus in the following way: Define the interval $\Lambda=[\frac{\pi}{16q^{2}},\frac{\pi}{q}-\frac{\pi}{16q^{2}}] $ and its subset $ \widetilde{\Lambda}=[\frac{\pi}{8q^{2}},\frac{\pi}{q}-\frac{\pi}{8q^{2}}]$ and impose the additional condition on the indexing set $K$  that $I_{k}\subset\widetilde{\Lambda}$ for each $k\in K$ and from now on  assume $\widetilde{I}_{2k}\subset\widetilde{\Lambda}$.

As per \eqref{dyn:56}, below, there exists constants $0<d_{-},d_{+}<\infty$ such that 
\begin{equation}\label{dyn:07.5}
d_{-}\ell_{j}  \leq  |\lambda'_{j}(\varkappa)|\leq d_{+}q^{2}\ell_{j}, \quad \forall\varkappa\in\Lambda. 
\end{equation} The upper bound will be used on the first term in \eqref{dyn:53} and the lower bound on the other two.

We have $\lambda_{j}^{-1}(E)=\frac{2k\pi}{nq}+s$ for some $-\frac{3\pi}{8nq}\leq s\leq\frac{3\pi}{8nq}$, since $\lambda_{j}^{-1}(E)\in\widetilde{I}_{2k}$. Moreover, by assumption we have $\varkappa\in\widetilde{I}_{2k}\subset\widetilde{\Lambda}\subset\Lambda$ so  \eqref{dyn:07.5}  and \eqref{dyn:57} imply $|\lambda_{j}(\varkappa)-E| \leq d_{+}q^{2}\ell_{j}|\varkappa-\frac{2k\pi}{nq}-s|$, therefore
\begin{equation}\label{dyn:58.5}
g_{j}(\varkappa) >  \frac{1}{\pi (d_{+}q^{2}\ell_{j})^{2}}\frac{1}{(\varkappa-\frac{2k\pi}{nq}-s)^{2}+(d_{+}q^{2}\ell_{j}T)^{-2}}\min_{\varkappa\in\widetilde{I}_{2k}}\varphi_{j}(\varkappa),\quad \forall \varkappa\in\widetilde{I}_{2k}
\end{equation}
and thus 
\begin{equation}\label{dyn:18}
\begin{split}
\int_{\widetilde{I}_{2k}}g_{j}(\varkappa)d\varkappa
&>\frac{1}{\pi}\frac{T}{d_{+}q^{2}\ell_{j}}\bigg(\arctan\bigg(d_{+}q^{2}\ell_{j}T\bigg(\frac{3\pi}{8nq}-s\bigg)\bigg)+\arctan\bigg(d_{+}q^{2}\ell_{j}T\bigg(\frac{3\pi}{8nq}+s\bigg)\bigg)\bigg)\min_{\varkappa\in\widetilde{I}_{2k}}\varphi_{j}(\varkappa)\\
&>\frac{1}{\pi}\frac{T}{d_{+}q^{2}\ell_{j}}\bigg(\frac{\pi}{2}-\frac{4n}{3\pi d_{+}q\ell_{j}T}\bigg)\min_{\varkappa\in\widetilde{I}_{2k}}\varphi_{j}(\varkappa) \\
&>\frac{1
}{4d_{+}}\frac{T}{q^{2}\ell_{j}}\min_{\varkappa\in\widetilde{I}_{2k}}\varphi_{j}(\varkappa).
\end{split}
\end{equation}
In the first inequality we used oddness of the arctangent. In the second inequality we lower-bounded the arctangents by their minima over $-\frac{3\pi}{8nq}\leq s\leq\frac{3\pi}{8nq}$, which happens to be minimal at the edges $s=\pm\frac{3\pi}{8nq}$, and then we applied the lower bound in 
\begin{equation}\label{dyn:59}
\frac{\pi}{2}-\frac{1}{x} < \arctan(x) < \frac{\pi}{2}-\frac{1}{x}+\frac{1}{3x^{3}},\quad \forall x>0.
\end{equation} The third inequality follows from the upper bound assumption on $n<c_{1}\eta q^{-4}\ell_{j}T$.

Now let us estimate the second quantity in \eqref{dyn:53}. Indeed, first recall that the function $g_{j}$ is positive in the even subintervals $I_{2k+2l}$ and negative in the odd subintervals $I_{2k+2l-1}$. We shall ignore the positive area obtained over the even intervals $I_{2k+2l}$ and shall only estimate all of the negative area obtained over the odd subintervals $I_{2k+2l-1}$. Namely, we shall only estimate the RHS of 
$$
\int_{\widetilde{I}_{2k}^{\mathsf{c}}}g_{j}(\varkappa)\,d\varkappa > \sum_{l}\int_{I_{2k+2l-1}}g_{j}(\varkappa)\,d\varkappa. 
$$
 Furthermore, we shall consider only the odd subintervals $I_{2k+2l-1}$ to the right ($l\geq1$) of $I_{2k}$, the argument for the intervals to the left of $I_{2k}$ is very similar. In particular, we shall estimate the quantity 
$$
\sum_{l\geq1}\int_{I_{2k+2l-1}}g_{j}(\varkappa)\,d\varkappa
$$ from below. We have the lower bound on the derivative \eqref{dyn:07.5}  when $I_{2k+2l-1}\subset\Lambda$, but do not have it when $I_{2k+2l-1}\not\subset\Lambda$, so we treat both cases separately. Let us start with the former.

The lower bound \eqref{dyn:07.5} gives  $\min_{\varkappa \in I_{2k+2l-1}}|\lambda'_{j}(\varkappa)|\geq d_{-}\ell_{j}$, since  $I_{2k+2l-1}\subset \Lambda$. And since $\lambda^{-1}_{j}(E)\leq \widetilde{b}_{2k}$ where $\widetilde{b}_{2k}=\frac{2k\pi}{nq}+\frac{3\pi}{8nq}$ is the right edge of the subinterval $\widetilde{I}_{2k}=[\widetilde{a}_{2k},\widetilde{b}_{2k}]$, we have 
\begin{equation}\label{dyn:58}
|\lambda_{j}(\varkappa)-E|\geq |\varkappa-\lambda^{-1}_{j}(E)|\min_{\varkappa \in \Lambda}|\lambda'_{j}(\varkappa)| \geq  d_{-} \ell_{j}|\varkappa-\widetilde{b}_{2k}|
\end{equation}
which implies 
\begin{equation}\label{dyn:67}
g_{j}(\varkappa)\leq \frac{2}{(d_{-}\ell_{j})^{2}}\frac{1}{(\varkappa-\widetilde{b}_{2k})^{2}+(d_{-}\ell_{j} T)^{-2}},\quad\forall \varkappa\in I_{2k+2l-1}
\end{equation}
for  any $l\geq 1$ such that $I_{2k+2l-1}\subset \Lambda$, having used $\varphi_{j}\leq 2$ and $|\cos|\leq 1$. The upper bound \eqref{dyn:67} implies 
\begin{equation}\label{dyn:20}
\int_{I_{2k+2l-1}}g_{j}(\varkappa)\,d\varkappa > -\frac{2T}{d_{-}\ell_{j}}A_{2k+2l-1}
\end{equation} 
where 
\begin{equation*}
A_{2k+2l-1}=\arctan\bigg(\frac{\pi d_{-}\ell_{j}T}{nq}\bigg(2l-\frac{7}{8}\bigg)\bigg)-\arctan\bigg(\frac{\pi d_{-}\ell_{j}T}{nq}\bigg(2l-\frac{15}{8}\bigg)\bigg)
\end{equation*}
having computed the definite integral, substituted the limits of integration and then simplified the resulting expression using $(2k+2l-1)\frac{\pi}{nq}\pm\frac{\pi}{2nq}-\widetilde{b}_{2k} = (2l-\frac{11}{8}\pm\frac{4}{8})\frac{\pi}{nq}$. 

We now apply both of the estimates on the arctangent in \eqref{dyn:59} to the term $A_{2k+2l-1}$, to obtain  
\begin{equation}\label{dyn:21}
A_{2k+2l-1}<\frac{nq}{\pi d_{-}\ell_{j}T}\bigg(\frac{1}{(2l-\frac{15}{8})(2l-\frac{7}{8})}+\bigg(\frac{nq}{\pi d_{-}\ell_{j}T}\bigg)^{2}\frac{1}{3(2l-\frac{7}{8})^{3}}\bigg) < \frac{c_{1}\eta }{\pi q^{3}d_{-}}R(l)
\end{equation}
having substituted $\frac{nq}{\pi d_{-}\ell_{j}T}<1 $ (which follows from $n<c_{1}\eta q^{-4}\ell_{j}T$) inside the bracket (to get the rational function $R(l)$) and substituted $n<c_{1}\eta q^{-4}\ell_{j}T$ outside the bracket to get the coefficient in the RHS. Since $R(l)$ decays quadratically, we have $\sum_{l\geq1}R(l)<\infty$, so \eqref{dyn:20} and \eqref{dyn:21} give  
\begin{equation}\label{dyn:22}
\sum_{l\geq 1;I_{2k+2l-1}\subset\Lambda}\int_{I_{2k+2l-1}}g_{j}(\varkappa)\,d\varkappa 
> -\frac{c_{1}\eta T}{q^{3}\ell_{j}}C_{2}.
\end{equation} 

We now turn to the more straightforward case that $I_{2k+2l-1}\not\subset\Lambda$ for $l\geq1$. Indeed, we must estimate the quantity $|\lambda_{j}(\varkappa)-E|$, from below, in an alternative way to \eqref{dyn:58}. Indeed, the lower bound \eqref{dyn:07.5} implies 
\begin{equation}\label{dyn:60}
|\lambda_{j}(\varkappa)-E|\geq \bigg|\lambda_{j}\bigg(\frac{\pi}{q}-\frac{\pi}{16q^{2}}\bigg)-\lambda_{j}\bigg(\frac{\pi}{q}-\frac{\pi}{8q^{2}}\bigg)\bigg|\geq \frac{d_{-}\ell_{j}\pi}{16q^{2}},\quad \forall\varkappa\in I_{2k+2l-1}
\end{equation}
since the eigenvalue $\lambda_{j}$ is monotonic, since $\lambda_{j}^{-1}(E)\leq \frac{\pi}{q}-\frac{\pi}{8q^{2}}$ and since $I_{2k+2l-1}\subset [\frac{\pi}{q}-\frac{\pi}{16q^{2}},\frac{\pi}{q}]$ (in the case that $I_{2k+2l-1}$ sits on the edge of $\Lambda$, one need only scale \eqref{dyn:60}, slightly), therefore
\begin{equation}\label{dyn:61}
g_{j}(\varkappa)\geq -\frac{2}{(\frac{d_{-}\ell_{j}\pi}{16q^{2}})^{2}+T^{-2}}>-2\frac{16^{2}q^{4}}{(d_{-}\ell_{j}\pi)^{2}},\quad \forall\varkappa\in I_{2k+2l-1}
\end{equation}
for each $l\geq1$ for which $I_{2k+2l-1}\not\subset\Lambda$, having used $\varphi_{j}\leq 2$ and $|\cos|\leq 1$. \eqref{dyn:61} implies  
\begin{equation}\label{dyn:23}
\sum_{l\geq1;I_{2k+2l-1}\not\subset\Lambda}\int_{I_{2k+2l-1}}g_{j}(\varkappa)\,d\varkappa 
>-2n|I_{2k+2l-1}|\frac{16^{2}q^{4}}{(d_{-}\ell_{j}\pi)^{2}}=-2\frac{16^{2}q^{3}}{(d_{-}\ell_{j})^{2}\pi}
\end{equation} since there are less than $n$ intervals satisfying $I_{2k+2l-1}\not\subset\Lambda$ and $|I_{2k+2l-1}|=\frac{\pi}{nq}$. 

For the third term in \eqref{dyn:53} we argue as follows. Since $\lambda_{j}^{-1}(E)\in\widetilde{I}_{2k}\subset\widetilde{\Lambda}$, by monotonicity of the eigenvalues we need only compare the eigenvalue at either of the two edges of $\Lambda,\widetilde{\Lambda}$, to obtain 
\begin{equation}\label{dyn:62}
\min_{i\neq j}|\lambda_{i}(\varkappa)-E|\geq  \min(|\lambda_{j}(0)-E|,|\lambda_{j}(\pi/q)-E|) \geq \frac{d_{-}\ell_{j}\pi}{16q^{2}},\quad \forall\varkappa\in[0,\pi/q]
\end{equation} where we may indeed have equality in the first inequality since we make no assumption on eigenvalue separation. Then, $\varphi_{j}\leq 2$,   $|\cos|\leq 1$ and \eqref{dyn:62}  imply 
\begin{equation}\label{dyn:63}
g_{i}(\varkappa) >-2\frac{16^{2}q^{4}}{(d_{-}\ell_{j}\pi)^{2}},\quad \forall\varkappa\in[0,\pi/q]
\end{equation} for each $i\neq j$, which implies  
\begin{equation}\label{dyn:23.7}
\sum_{i\neq j}\int_{0}^{\frac{\pi}{q}}g_{i}(\varkappa)\,d\varkappa 
>-2\frac{16^{2}q^{4}}{(d_{-}\ell_{j})^{2}\pi}
\end{equation} since the sum on the LHS has $q-1$ terms.

Finally, let us combine the estimates \eqref{dyn:18}-\eqref{dyn:23.7} to obtain a lower bound on $|h(E)|$. First, multiply \eqref{dyn:22}  and \eqref{dyn:23} by a factor of $2$ to account for the subintervals to the left of the $2k$-th one. All of the estimates \eqref{dyn:18}-\eqref{dyn:23.7} also hold for $E\in\lambda_{j}(\widetilde{I}_{k})$ for odd $k\in K$, with opposite signs (as mentioned previously, we are squaring $h(E)$ so it makes no difference). Combining all of the lower bounds \eqref{dyn:18}-\eqref{dyn:23.7}, we have: For sufficiently small $c_{1}>0$ and sufficiently large $C<\infty$, for any $E\in\lambda_{j}(\widetilde{I}_{k})$ and $k\in K$,  if $T>\frac{C}{c_{1}}\eta^{-2}q^{8}\ell_{j}^{-2}+1$, then 
\begin{equation}\label{dyn:23.5}
|h(E)|>\frac{T}{q^{2}\ell_{j}}\bigg(\frac{1}{4d_{+}}\min_{\varkappa\in\widetilde{I}_{k}}\varphi_{j}(\varkappa)-2c_{1}C_{2}\frac{\eta}{q}-6\frac{16^{2}q^{6}}{\pi d^{2}_{-}\ell_{j}T}\bigg)>\frac{c_{2}\eta T}{q^{3}\ell_{j}}
\end{equation} for $C\eta^{-1}q^{4}\ell_{j}^{-1}<n<c_{1}\eta q^{-4}\ell_{j}T$. In \eqref{dyn:23.5}, we used the lower bound $\min_{\varkappa\in\widetilde{I}_{k}}\varphi_{j}(\varkappa)>\frac{\eta}{2q}$,
which follows from the lower bound assumption on $n>C\eta^{-1}q^{4}\ell_{j}^{-1}$: Indeed, recall that $\max_{\varkappa\in I_{k}}\varphi_{j}(\varkappa)>\frac{\eta}{q}$ and we shall argue below, that 
\begin{equation}\label{dyn:07.6}
\max_{\varkappa\in\Lambda}|\varphi'_{j}(\varkappa)|\leq \frac{C_{3}q^{4}}{\ell_{j}}
\end{equation}
thus \eqref{dyn:23.5} follows for large $C<\infty$:
$$|\min_{\varkappa\in\widetilde{I}_{k}}\varphi_{j}(\varkappa)-\max_{\varkappa\in I_{k}}\varphi_{j}(\varkappa)| \leq | I_{k}|\max_{\varkappa\in\Lambda}|\varphi'_{j}(\varkappa)| \leq \frac{\pi C_{3}q^{3}}{\ell_{j} n}<\frac{\pi C_{3}\eta}{C q}.$$

Recalling \eqref{dyn:14} and \eqref{dyn:55}, 
\begin{equation}\label{dyn:68}
P_{q,T}(nq) > \frac{q^{2}}{\pi^{3}T^{3}}\#K\min_{k\in K}|\lambda_{j}(\widetilde{I}_{k})|\min_{E\in\lambda_{j}(\widetilde{I}_{k})}(h(E))^{2} 
\end{equation}
then Claim \ref{dyn:chebyshev}, below, gives 
$$
\#K\geq (|\{\varkappa:\varphi_{j}(\varkappa)>\eta/q\}|-|\widetilde{\Lambda}^{\mathsf{c}}|)/|I_{k}|>\frac{n}{4q}
$$ 
and \eqref{dyn:07.5} gives 
$$
|\lambda_{j}(\widetilde{I}_{k})|\geq |\widetilde{I}_{k}| \min_{\varkappa\in\Lambda}|\lambda'_{j}(\varkappa)| \geq \frac{3\pi d_{-}\ell_{j}}{4nq}
$$
with which the lemma follows from \eqref{dyn:23.5} and \eqref{dyn:68}.

\emph{Sketch of the proof of estimates \eqref{dyn:07.5}.} The characteristic polynomial of the Floquet matrix \eqref{dyn:floquet} is given by 
\begin{equation}\label{dyn:06}
D_{\varkappa,q}(E)=\det(A_{q}(\varkappa)-E) = \Delta_{q}(E)+2(-1)^{q-1}\cos(q\varkappa)
\end{equation}
where the discriminant $\Delta_{q}(E)$ is a polynomial of degree $q$ with real coefficients. Last \cite[Lemma 1]{La94} proves \eqref{dyn:07} for the discriminant $D_{\frac{\pi}{2q},q}$.  The arguments of Last can also be repeated for the characteristic polynomials $D_{\varkappa,q}$, for each $\varkappa\in(0,\frac{\pi}{q})$. By doing so, one obtains:  
\begin{equation}\label{dyn:07}
(1+\sqrt{5})(1-|\cos(q\varkappa)|)\leq \ell_{j}|D_{\varkappa,q}'(\lambda_{j}(\varkappa))|\leq e|D_{\varkappa,q}(\lambda_{j}(0))-D_{\varkappa,q}(\lambda_{j}(\pi/q))|
\end{equation} which holds for every $\varkappa\in(0,\frac{\pi}{q})$ and $j=1,\dots,q$. Evaluating the characteristic polynomial \eqref{dyn:06} at the eigenvalue $\lambda_{j}(\varkappa)$ and then differentiating with respect to $\varkappa\in\T_{q}$, gives  
$$
|\Delta_{q}'(\lambda_{j}(\varkappa))||\lambda_{j}'(\varkappa)|=2q|\sin(q\varkappa)|
$$
and since $\frac{d}{dE}\Delta_{q}=\frac{d}{dE}D_{\varkappa,q}$ for any $\varkappa\in\T_{q}$,  \eqref{dyn:07} also holds for  $|\Delta_{q}'(\lambda_{j}(\varkappa))|$, therefore
\begin{equation}\label{dyn:56}
2q\sin(q\varkappa) \frac{\ell_{j}}{4e}  \leq  |\lambda'_{j}(\varkappa)|\leq \frac{2q\sin(q\varkappa)}{1-|\cos(q\varkappa)|}\frac{\ell_{j}}{1+\sqrt{5}},\quad \forall\varkappa\in\bigg(0,\frac{\pi}{q}\bigg).
\end{equation}  By evaluating the LHS and RHS of the estimates \eqref{dyn:56} at the edge $\varkappa=\frac{\pi}{16q^{2}}$, one obtains \eqref{dyn:07.5}.  

\emph{Proof of \eqref{dyn:07.6}.} Expressing the function $\varphi_{j}$ as a sum of the squares of its real and imaginary part, taking the derivative followed by an application of the Cauchy-Schwarz inequality, gets us 
\begin{equation}\label{dyn:70}
|\varphi'_{j}(\varkappa)|\leq 2|\dot{\Psi}^{(j)}_{\varkappa}(0)||\Psi^{(j)}_{\varkappa}(0)|+2|\dot{\Psi}^{(j)}_{\varkappa}(1)||\Psi^{(j)}_{\varkappa}(1)|
\end{equation} where $\dot{\Psi}^{(j)}$ denotes the component-wise derivative of the eigenvector $\Psi^{(j)}$, with respect to $\varkappa$. 

The eigenvalues of the Floquet matrix $A_{q}(\varkappa)$ are simple in the interior $\varkappa\in(0,\pi/q)$; therefore, by perturbation theory one obtains the formula for the derivative
\begin{equation}\label{dyn:16}
\dot{\Psi}^{(j)}_{\varkappa}
= -\sum_{k\neq j}\frac{\langle\Psi^{(k)}_{\varkappa},\dot{A}_{q}(\varkappa)\Psi^{(j)}_{\varkappa}\rangle}{\lambda_{k}(\varkappa)-\lambda_{j}(\varkappa)}\Psi^{(k)}_{\varkappa}.
\end{equation}

The estimate \eqref{dyn:07.6} follows from \eqref{dyn:56}-\eqref{dyn:16}, as follows. An application of the Cauchy-Schwarz inequality followed by estimating the absolute value of the denominator of  \eqref{dyn:16} from below by integrating the lower bound of \eqref{dyn:56} from either edge $0,\frac{\pi}{q}$ of the half-torus until the point $\varkappa\in(0,\frac{\pi}{q})$, combined with \eqref{dyn:70}, gets us 
$$
|\varphi'_{j}(\varkappa)|\leq \frac{8eq^{2}\ell_{j}^{-1}}{1-|\cos(q\varkappa)|}
$$ the RHS of which is to be evaluated at the edge $\varkappa=\frac{\pi}{16q^{2}}$, to get \eqref{dyn:07.6}. 

\begin{claim}\label{dyn:chebyshev}
There exists $1\leq j\leq q$ such that $B_{q}^{(j)}\cap I\neq\varnothing$ and $|\{\varkappa:\varphi_{j}(\varkappa)>\eta/q\}|>\frac{\pi}{2q^{2}}$.
\end{claim}
\emph{Proof.} First note that 
$
\mu_{\varkappa,q}(I) = \sum_{\lambda_{j}(\varkappa)\in I} \varphi_{j}(\varkappa)\leq \sum_{j\in J}\varphi_{j}(\varkappa)
$ where $J=\{j:B_{q}^{(j)}\cap I\neq\varnothing\}$. Towards a contradiction let us suppose that the conclusion of the claim is false. Then for every $j\in J$ we have $|\{\varkappa:\varphi_{j}(\varkappa)\leq\eta/q\}|\geq\frac{\pi}{q}-\frac{\pi}{2q^{2}}$, so $\#J\leq q$  implies 
$$
\bigg|\bigcap_{j\in J}\{\varkappa:\varphi_{j}(\varkappa)\leq\eta/q\}\bigg|\geq\frac{\pi}{q} - \frac{\pi\#J}{2q^{2}} \geq \frac{\pi}{2q}.
$$
 Then $\bigcap_{j\in J}\{\varkappa:\varphi_{j}(\varkappa)\leq\eta/q\}\subseteq\{\varkappa:\sum_{j\in J}\varphi_{j}(\varkappa)\leq\eta\}$ contradicts $\inf_{\varkappa\in[0,\frac{\pi}{q}]}\mu_{\varkappa,q}(I)>\eta$.
\end{proof}

\emph{On Lemma \ref{dyn:03.5}.} For any bounded Schr\"odinger operator $H:\ell^{2}(\Z)\rightarrow \ell^{2}(Z)$, for any  $\psi,\phi\in\ell^{2}(\Z)$ and $T>0$, one obtains the remarkable identity 
\begin{equation*}
\frac{2}{T}\int_{0}^{\infty}|\langle\phi,e^{-itH}\psi\rangle|^{2}e^{-\frac{2t}{T}}dt = \frac{1}{\pi T}\int_{\R}|\langle \phi,(H-E-iT^{-1})^{-1}\psi \rangle|^{2} \,dE 
\end{equation*} 
by applying Plancherel's theorem: 
$$\int_{\R}|\hat{g}(\xi)|^{2}\,d\xi=\int_{\R}|g(t)|^{2}\,dt, \text{ where }\hat{g}(\xi)=\int_{\R}g(t)e^{-2\pi i\xi t}\,dt$$ for any $g\in L^{1}(\R)\cap L^{2}(\R)$ to the function 
$g(t) = \langle\phi,e^{-itH-tT^{-1}}\psi\rangle\chi_{[0,\infty)}(t)$, while applying the identity 
$$
\langle\phi,e^{-itH-tT^{-1}-2i\pi t\xi}\psi\rangle=\int_{\R}e^{-it\lambda-tT^{-1}-2i\pi t\xi}d\mu_{\phi,\psi}(\lambda)
$$ (see \eqref{dyn:04}, below) and scaling $\xi$ appropriately.

The periodic operator $H_{q}$ is diagonalizable in the Fourier space $L^{2}(\T_{q}\mapsto\C^{q},\langle \Psi,\Phi \rangle_{L^{2}})$ where $\T_{q}= \R/\frac{2\pi}{q}\Z$ and 
$\langle \Psi,\Phi \rangle_{L^{2}} = \sum_{|l|\leq\frac{q}{2}}\int_{\T_{q}} \overline{\Psi_{\varkappa}(l)}\Phi_{\varkappa}(l)\frac{d\varkappa}{\frac{2\pi}{q}}$. Indeed, we construct below, two unitary operators $U_{1,q}:\ell^{2}(\Z)\rightarrow\ell^{2}(\Z\mapsto\C^{q})$ and $U_{2,q}:\ell^{2}(\Z\mapsto\C^{q})\rightarrow L^{2}(\T_{q}\mapsto\C^{q})$ such that $$U_{q}H_{q}=M_{q}U_{q}$$ where $U_{q}=U_{2,q}U_{1,q}: \ell^{2}(\Z)\rightarrow L^{2}(\T_{q}\mapsto\C^{q})$ denotes the block Fourier transform and $M_{q}$ the multiplication operator acting as a multiplication by the Floquet matrix \eqref{dyn:floquet}.

For $e_{0}=(0,\dots,1,\dots,0)\in\C^{q}$, unitary equivalence and $U_{q}\delta_{nq}=e_{0}e_{n,q}\in L^{2}(\T_{q}\mapsto \C^{q})$ imply 
\begin{equation}\label{dyn:08}
\langle \delta_{nq},(H_{q}-E-iT^{-1})^{-1}\delta_{0} \rangle 
= \langle e_{0}e_{n,q},(M_{q}-E-iT^{-1})^{-1}e_{0}e_{0,q} \rangle
\end{equation}
where $e_{n,q}(\varkappa)=e^{inq\varkappa}$ is the $n$-th canonical basis vector for the Fourier space $L^{2}(\T_{q}\mapsto\C)$. Then expressing the vector $e_{0}$ in terms of the orthonormal eigenvectors $\{\Psi^{(j)}_{\varkappa}\}_{j=1}^{q}$ of the Floquet matrix gives 
$$
\langle e_{0}e_{n,q},(M_{q}-E-iT^{-1})^{-1}e_{0}e_{0,q} \rangle 
= \sum_{j=1}^{q}\int_{0}^{\frac{2\pi}{q}}\frac{e^{-inq\varkappa}|\langle\Psi_{\varkappa}^{(j)},e_{0}\rangle|^{2}}{\lambda_{j}(\varkappa)-E-iT^{-1}} \frac{d\varkappa}{\frac{2\pi}{q}} 
=  \sum_{j=1}^{q}\int_{0}^{\frac{\pi}{q}}\frac{\cos(nq\varkappa)|\langle\Psi_{\varkappa}^{(j)},e_{0}\rangle|^{2}}{\lambda_{j}(\varkappa)-E-iT^{-1}} \frac{d\varkappa}{\frac{\pi}{q}}
$$
since $(\lambda_{j}(-\varkappa),\Psi^{(j)}_{-\varkappa})=(\lambda_{j}(\varkappa),\overline{\Psi^{(j)}_{\varkappa}})$, seen by transposing the Floquet matrix. 

Let $U_{1,q}:\ell^{2}(\Z)\rightarrow\ell^{2}(\Z\mapsto\C^{q})$ be the unitary operator taking blocks (of length $q$) of the sequence of Fourier coefficients $\hat{\psi}\in\ell^{2}(\Z)$ to a single component $\hat{\Psi}_{n}\in\C^{q}$ of a vector-valued sequence $\hat{\Psi}\in\ell^{2}(\Z\mapsto\C^{q})$ of Fourier coefficients. Namely, $$U_{1,q}\hat{\psi} =\hat{\Psi}=((\hat{\Psi}_{n}(l))_{|l|\leq\frac{q}{2}})_{n\in\Z}=((\hat{\psi}(nq+l))_{|l|\leq\frac{q}{2}})_{n\in\Z}$$ where $\hat{\Psi}_{n}(l)=\hat{\psi}(nq+l)\in\C$ denotes the $l$-th component of the vector $\hat{\Psi}_{n}\in\C^{q}$, which is itself the $n$-th component of the sequence $\hat{\Psi}\in\ell^{2}(\Z\mapsto\C^{q})$. 

Let $U_{2,q}:\ell^{2}(\Z\mapsto\C^{q})\rightarrow L^{2}(\T_{q}\mapsto\C^{q})$ be the Fourier transform taking the vector-valued sequence of Fourier coefficients $\hat{\Psi}\in\ell^{2}(\Z\mapsto\C^{q})$ to its corresponding function in the vector-valued Fourier space $L^{2}(\T_{q}\mapsto\C^{q})$. Namely, $$U_{2,q}\hat{\Psi}=\Psi=\sum_{n\in\Z}\hat{\Psi}_{n}e_{n,q}.$$ The function $\Psi\in L^{2}(\T_{q}\mapsto\C^{q})$ is vector-valued. We adopt the notation $\Psi_{\varkappa}$ to mean the function $\Psi$ evaluated at the point $\varkappa\in\T_{q}$. $\Psi_{\varkappa}$ is a vector in $\C^{q}$ which has components which we denote by $\Psi_{\varkappa}(l)$ and satisfy $\Psi_{\varkappa}(l)=\sum_{n\in\Z}\hat{\Psi}_{n}(l)e^{inq\varkappa}$.

\section{Proof of the theorem}

Lemma \ref{dyn:uniformweakintervals}, below, verifies the assumption of Lemma \ref{dyn:periodiclower} by ensuring that the canonical  spectral measures of the Floquet matrix evaluated at the vicinity of a minimum of the Lyapunov exponent on the spectrum are uniformly bounded from below in both the period and the phase. The general idea is to approximate the canonical  spectral measure $\mu_{\alpha,\theta}$ associated with the limiting quasiperiodic Schr\"odinger operator $H_{\alpha,\theta}$ by the canonical  spectral measures of the Floquet matrix $\mu_{\alpha_{m},\theta}^{(\varkappa)}$ where  $\alpha_{m}=\frac{p_{m}}{q_{m}}\rightarrow\alpha$ and combine this with the fact that the topological support of the canonical  spectral measure coincides with the spectrum of the operator $\text{supp}(\mu_{\alpha,\theta})=\sigma_{\alpha}$.

The spectral measure associated with a discrete self-adjoint one-dimensional bounded Schr\"odinger operator $H:\ell^{2}(\Z)\rightarrow\ell^{2}(\Z)$ is the complex Borel measure $\mu_{\phi,\psi}$ for which
\begin{equation}\label{dyn:04}
\langle\phi,g(H)\psi\rangle = \int_{\R}g(\lambda)\,d\mu_{\phi,\psi}(\lambda)
\end{equation}
holds for all compactly supported, bounded Borel measurable functions $g:\R\rightarrow\C$. 
The spectral measures $\mu_{\phi,\psi}$ are positive probability measures in the case that $\psi=\phi$. The topological support of a Borel measure $\mu$ on the real line is defined as 
$$
\text{supp}(\mu) = \{\lambda\in\R:\mu((\lambda-\varepsilon,\lambda+\varepsilon))>0,\, \forall \varepsilon>0\}.
$$ 
Let $\delta_{k}\in\ell^{2}(\Z)$ denote the $k$-th canonical basis vector of $\ell^{2}(\Z)$. In general it is not true that $\text{supp}(\mu_{\delta_{k},\delta_{k}})=\sigma(H)$, but indeed the canonical  spectral measure $\mu = \mu_{\delta_{0},\delta_{0}}+\mu_{\delta_{1},\delta_{1}}$ associated with $H$ satisfies 
$$
\sigma(H)=\text{supp}(\mu).
$$

\begin{lemma}\label{dyn:uniformweakintervals} 
Let $H_{\alpha,\theta}$ be as in \eqref{dyn:01} with  associated continuous Lyapunov exponent $\gamma$. Let $\alpha_{m}=\frac{p_{m}}{q_{m}}\rightarrow\alpha$ be a sequence of rationals and let $B_{\varepsilon}(E_{0})\subset\R$ denote a ball of radius $\varepsilon>0$ centred at $E_{0}\in\R$. Fix $\varepsilon>0$, $E_{0}\in\sigma_{\alpha}$, let $\beta>2\sup_{E\in B_{\varepsilon}(E_{0})}\gamma$ and suppose $|\alpha-\alpha_{m}|<e^{-\beta q_{m}}$ for every $m\geq1$, then 
\begin{enumerate}[(i)]
\item $\lim_{m\rightarrow\infty}\sup_{\theta\in\T}\sup_{\varkappa\in[0,\frac{\pi}{q_{m}}]}\big| \mu_{\alpha_{m},\theta}^{(\varkappa)}(B_{\varepsilon}(E_{0}))-  \mu_{\alpha,\theta}(B_{\varepsilon}(E_{0}))\big|=0$. 
\item $\mu_{\alpha,\theta}(B_{\varepsilon}(E_{0}))$ is continuous in  $\theta\in\T$. 
\end{enumerate}
\end{lemma}

\begin{rmk}
The proof of Lemma \ref{dyn:uniformweakintervals} requires upgrading weak convergence to convergence on intervals. This requires that the limiting measure be non-atomic on the boundary of the interval. This fact (Lemma \ref{dyn:nonatomic}) is a variant of Gordon's lemma \cite[Theorem 10.3]{CFKS09}. The conclusion of Lemma \ref{dyn:uniformweakintervals} also holds under the  assumption $\beta>\sup_{E\in B_{\varepsilon}(E_{0})}\gamma$. Neither the factor of $3$ in the lower bound assumption on $\beta(\alpha)$ nor the assumption $\delta<\frac{1}{2}$ originate from this factor of $2$ from Gordon's argument. The proof of Lemma \ref{dyn:uniformweakintervals} is provided in Section \ref{dyn:secuniflowerbound} where we also provide a proof of Lemma \ref{dyn:nonatomic}.
\end{rmk}

Another key ingredient in the proof of the theorem is Proposition \ref{dyn:bandwidths}, which bounds the bandwidths of the periodic operator $H_{\alpha_{m},\theta}$, from below, in terms of the Lyapunov exponent associated with the limiting quasiperiodic operator $H_{\alpha,\theta}$.

\begin{prop}[\cite{Ha22}]\label{dyn:bandwidths}
Let $H_{\alpha,\theta}$ be a bounded discrete one-dimensional Schr\"odinger operator \eqref{dyn:01} with $\theta\in\mathbb{T}$, $\alpha\in\R\setminus\Q$ and associated continuous Lyapunov exponent $\gamma$. Let $\alpha_{m}=\frac{p_{m}}{q_{m}}\rightarrow\alpha$ be a sequence of rationals. Let  $B^{(j)}_{\alpha_{m},\theta}$ denote the $j$-th band in the spectrum of the periodic operator $H_{\alpha_{m},\theta}$. We have  
$$
\liminf_{m\rightarrow \infty}\min_{j\in[1,q_{m}], \theta\in\T}(q_{m}^{-1}\log |B^{(j)}_{\alpha_{m},\theta}|+\gamma(b^{(j)}_{\alpha_{m},\theta})) \geq 0
$$ where $b^{(j)}_{\alpha_{m},\theta}$ is the centre of the band $B^{(j)}_{\alpha_{m},\theta}$. 
\end{prop}

\begin{proof}[Proof of Theorem \ref{dyn:theorem}]
Let us first show that Lemma \ref{dyn:uniformweakintervals} implies the conclusion of the Lemma \ref{dyn:periodiclower} in the current setting. Let $\alpha_{m}=\frac{p_{m}}{q_{m}}$ be the sequence of canonical convergents associated with $\alpha\in\R\setminus\Q$. Fix a point in the spectrum $E_{0}\in\sigma_{\alpha}$ which minimises the Lyapunov exponent $\gamma_{0}=\gamma(E_{0}) =\min_{\sigma_{\alpha}}\gamma$. If $\beta(\alpha)>3\delta^{-1}\gamma_{0}$ then $\frac{\beta}{3}=\delta^{-1}(\gamma_{0}+2\varepsilon')$ (for sufficiently small  $\varepsilon'>0$) satisfies $\beta(\alpha)>\beta>3\delta^{-1}\gamma_{0}$ and there exists a subsequence $m_{k}$ for which $q_{m_{k}}^{-1}\log q_{m_{k}+1}>\beta$ and hence $|\alpha-\alpha_{m_{k}}|<\frac{1}{q_{m_{k}}q_{m_{k}+1}}<e^{-q_{m_{k}}\beta}$ for all sufficiently large $k$. Note also that $\beta>2\sup_{B_{\varepsilon}(E_{0})}\gamma$, for sufficiently small $\varepsilon>0$, by continuity of the Lyapunov exponent $\gamma$. It follows from part $(i)$ of Lemma \ref{dyn:uniformweakintervals}  that there exists $k_{0}(E_{0},\varepsilon)$ such that for all $k>k_{0}$, 
$$
\inf_{\theta\in\T}\inf_{\varkappa\in[0,\frac{\pi}{q_{m_{k}}}]} \mu_{\alpha_{m_{k}},\theta}^{(\varkappa)}(B_{\varepsilon}(E_{0})) > \frac{1}{2}\min_{\theta\in\T}\mu_{\alpha,\theta}(B_{\varepsilon}(E_{0}))
$$ then $E_{0}\in\sigma_{\alpha}$ and part $(ii)$ of Lemma \ref{dyn:uniformweakintervals} imply that $\mu_{\alpha,\theta}(B_{\varepsilon}(E_{0}))$ is a strictly positive continuous function of $\theta\in\T$ so for every $\varepsilon>0$ there exists $\eta=\eta(\varepsilon,E_{0})>0$ such that $\min_{\theta\in\T}\mu_{\alpha,\theta}(B_{\varepsilon}(E_{0}))>2\eta$ hence 
\begin{equation}\label{dyn:uniflowerbound}
\inf_{\theta\in\T}\inf_{\varkappa\in[0,\frac{\pi}{q_{m_{k}}}]}\mu_{\alpha_{m_{k}},\theta}^{(\varkappa)}(B_{\varepsilon}(E_{0}))>\eta
\end{equation} for $k>k_{0}$. Assuming $k>k_{0}$, the uniform lower bound  \eqref{dyn:uniflowerbound} is precisely the assumption of Lemma \ref{dyn:periodiclower} so for $0<c,c_{1},C<\infty$ as in the conclusion of Lemma \ref{dyn:periodiclower}: For each $\theta\in\T$  there exists a band $B_{\alpha_{m_{k}},\theta}^{(j_{\theta})}\cap B_{\varepsilon}(E_{0})\neq\varnothing$ with length $\ell_{\theta}=|B_{\alpha_{m_{k}},\theta}^{(j_{\theta})}|$, such that for $T>\frac{C}{c_{1}}\eta^{-2}q_{m_{k}}^{8}\ell_{\theta}^{-2}+1$, 
\begin{equation}\label{dyn:26}
P_{\alpha_{m_{k}},\theta,T}(q_{m_{k}}n)>\frac{c\eta^{2}}{q_{m_{k}}^{6}\ell_{\theta}T},  \text{ for }C\eta^{-1}q_{m_{k}}^{4}\ell_{\theta}^{-1}<n<c_{1}\eta q_{m_{k}}^{-4}\ell_{\theta}T
\end{equation}
 where $P_{\alpha_{m_{k}},\theta,T}$ is as in \eqref{dyn:10}.  

The term $\ell_{\theta}$ can be controlled uniformly, by means of Proposition \ref{dyn:bandwidths} in the following way. For all sufficiently large $k$ we have $B_{\alpha_{m_{k}},\theta}^{(j_{\theta})}\subset B_{2\varepsilon}(E_{0})$ for all $\theta\in\T$, since $\max_{j}|B_{\alpha_{m_{k}},\theta}^{(j)}|\leq\frac{2\pi}{q_{m_{k}}}$ (see e.g.\ \cite{Ha22}). Continuity of the Lyapunov exponent and Proposition \ref{dyn:bandwidths} imply, for all sufficiently large $k$,
\begin{equation}\label{dyn:31}
1/\inf_{\theta}\ell_{\theta}<e^{(\gamma_{0}+\varepsilon'')q_{m_{k}}}
\end{equation} where we set $0<\varepsilon''<\varepsilon'$.

We shall show below, that on the subsequence  
$T_{m_{k}}=e^{\delta^{-1}(\gamma_{0}+\varepsilon')q_{m_{k}}}$,  we have, for any $\theta\in\T$,  
\begin{equation}\label{dyn:30}
P_{\alpha,\theta,T_{m_{k}}}(q_{m_{k}}n) > \frac{1}{2}\frac{c\eta^{2}}{q_{m_{k}}^{6}\ell_{\theta}T_{m_{k}}}, \quad C\eta^{-1}q_{m_{k}}^{4}\ell_{\theta}^{-1}<\frac{1}{2}c_{1}\eta q_{m_{k}}^{-4}\ell_{\theta}T_{m_{k}}<n <c_{1}\eta q_{m_{k}}^{-4}\ell_{\theta}T_{m_{k}}
\end{equation} 
from which the conclusion of the theorem follows, for sufficiently large $k$:  
\begin{equation}\label{dyn:32}
\begin{split}
\min_{\theta\in\T}  M_{\alpha,\theta,p} (T_{m_{k}})
&>  \inf_{\theta\in\T}\sum_{\frac{1}{2}c_{1}\eta q_{m_{k}}^{-4}\ell_{\theta}T_{m_{k}}< n< c_{1}\eta q_{m_{k}}^{-4}\ell_{\theta}T_{m_{k}}}(q_{m_{k}}n)^{p}P_{\alpha,\theta,T_{m_{k}}}(q_{m_{k}}n) \\
&> \frac{c_{1}c\eta^{3}}{4q_{m_{k}}^{10}}\bigg(\frac{1}{2}c_{1}\eta q_{m_{k}}^{-3}T_{m_{k}}\inf_{\theta}\ell_{\theta}\bigg)^{p}>\frac{c_{1}c\eta^{3}}{4q_{m_{k}}^{10}} T_{m_{k}}^{(1-\delta)p}, \quad \forall p>0
\end{split}
\end{equation} having used \eqref{dyn:31} on the final inequality. 

We now return to \eqref{dyn:30}. Indeed, first note that for any bounded self-adjoint operators $H_{1},H_{2}$,
\begin{equation}\label{dyn:78}
||\langle\delta_{n},e^{-itH_{1}}\delta_{0}\rangle|^{2}-|\langle\delta_{n},e^{-itH_{2}}\delta_{0}\rangle|^{2}|\leq 2|\langle\delta_{n},e^{-itH_{1}}\delta_{0}\rangle-\langle\delta_{n},e^{-itH_{2}}\delta_{0}\rangle|.
\end{equation}

Let $L$ denote the Lipschitz constant of $f$, we show that the RHS of  \eqref{dyn:78} is bounded above by 
\begin{equation}\label{dyn:73}
|\langle\delta_{q_{m_{k}}n},(e^{-itH_{\alpha,\theta}}-e^{-itH_{\alpha_{m_{k}},\theta}})\delta_{0}\rangle|<\varepsilon_{1}(t)=C'Lt^{2}e^{-\beta q_{m_{k}}}+2C_{1}e^{-c_{3}\max(|q_{m_{k}}n|,t)}
\end{equation}
and in particular 
\begin{equation}\label{dyn:79}
|\langle\delta_{q_{m_{k}}n},e^{-itH_{\alpha,\theta}}\delta_{0}\rangle|^{2} > |\langle\delta_{q_{m_{k}}n},e^{-itH_{\alpha_{m_{k}},\theta}}\delta_{0}\rangle|^{2} -  2{\varepsilon}_{1}(t). 
\end{equation}   

Indeed, for any bounded  Schr\"odinger operator $H:\ell^{2}(\Z)\rightarrow\ell^{2}(\Z)$, there exist constants $C',C_{1},c_{3}$, depending only on the norm $\|H\|$ such that if $N=C't$, then 
\begin{equation}\label{dyn:27}
|\langle\delta_{n},e^{-itH}\delta_{0}\rangle-\langle\delta_{n},e^{-itH_{N}}\delta_{0}\rangle|< C_{1}e^{-c_{3}\max(|n|,t)},\quad \forall\, n\in\Z, \,t\geq 0
\end{equation} where $H_{N}$ denotes the restriction of the operator $H$ to the finite interval $[-N,N]\subset\Z$, with Dirichlet boundary condition. The proof of \eqref{dyn:27} is standard and we provide it Section \ref{dyn:secuniflowerbound}, for completeness. It follows from the triangle inequality as well as \eqref{dyn:27}, that 
\begin{equation}\label{dyn:72}
|\langle\delta_{q_{m_{k}}n},(e^{-itH_{\alpha,\theta}}-e^{-itH_{\alpha_{m_{k}},\theta}})\delta_{0}\rangle| < |\langle\delta_{q_{m_{k}}n},(e^{-itH_{\alpha,\theta,N}}-e^{-itH_{\alpha_{m_{k}},\theta,N}})\delta_{0}\rangle|+ 2C_{1}e^{-c_{3}\max(|q_{m_{k}}n|,t)}. 
\end{equation} 
For Hermitian matrices $A$ and $B$ one obtains $\|e^{-itA}-e^{-itB}\|\leq \min(2,t\|A-B\|)$ by computing $\frac{d}{dt}(I-e^{itB}e^{-itA})$ (e.g.\ by series expansion one shows $\frac{d}{dt}e^{itB}=iBe^{itB}=ie^{itB}B$) and subsequently applying the fundamental theorem of calculus. \eqref{dyn:73} follows from the Cauchy-Schwarz inequality and $\|H_{\alpha,\theta,N}-H_{\alpha_{m_{k}},\theta,N}\| = \|V_{\alpha,\theta,N}-V_{\alpha_{m_{k}},\theta,N}\| = \max_{|n|\leq N}|f(n\alpha+\theta)-f(n\alpha_{m_{k}}+\theta)| \leq LNe^{-\beta q_{m_{k}}}$.

\eqref{dyn:79} also holds for the entry $(1,q_{m_{k}}n+1)$. Summing the entries and Abel-averaging then gives  
\begin{equation}\label{dyn:77}
P_{\alpha,\theta,T}(q_{m_{k}}n)
>P_{\alpha_{m_{k}},\theta,T}(q_{m_{k}}n)-\frac{2}{T}\int_{0}^{\infty}4\varepsilon_{1}(t)e^{-\frac{2t}{T}}\, dt.
\end{equation}
By direct computation of the integral in \eqref{dyn:77} while assuming $n>\frac{1}{2}c_{1}\eta q_{m_{k}}^{-4}\ell_{\theta}T$, gives  
\begin{equation}\label{dyn:28.5}
P_{\alpha,\theta,T}(q_{m_{k}}n)
>P_{\alpha_{m_{k}},\theta,T}(q_{m_{k}}n)-\varepsilon_{2}(T), \quad \varepsilon_{2}(T)=C_{3}e^{-\beta q_{m_{k}}}T^{2}+C_{3}e^{-\frac{1}{2}c_{1}\eta c_{3}q_{m_{k}}^{-3}\ell_{\theta}T}. 
\end{equation}

In order to get \eqref{dyn:30}, we need to show that the error $\varepsilon_{2}$ is less than half of our lower bound \eqref{dyn:26} on $P_{\alpha_{m_{k}},\theta,T}(q_{m_{k}}n)$, uniformly in $\theta\in\T$, on the subsequence $T_{m_{k}}$, namely, 
\begin{equation}\label{dyn:74}
\sup_{\theta\in\T}\varepsilon_{2}(T_{m_{k}})<\frac{1}{2}\inf_{\theta\in\T}\frac{c\eta^{2}}{q_{m_{k}}^{6}\ell_{\theta}T_{m_{k}}}.
\end{equation}
We shall first show that the second term of $\varepsilon_{2}(T_{m_{k}})$ is less than the first term, then we show that the first term is bounded by the infimum in \eqref{dyn:74}.  

Rewrite
\begin{equation}\label{dyn:76}
\varepsilon_{2}(T_{m_{k}})=C_{3}e^{-\beta q_{m_{k}}}T_{m_{k}}^{2}+C_{3}e^{-\frac{1}{2}c_{1}\eta c_{3}q_{m_{k}}^{-3}\ell_{\theta}T_{m_{k}}},
\end{equation}
we claim that \eqref{dyn:74} follows from the fact that the subsequence $T_{m_{k}} = e^{\delta^{-1}(\gamma_{0}+\varepsilon')q_{m_{k}}}$ satisfies 
\begin{equation}\label{dyn:75}
\frac{2\beta}{c_{1}c_{3}\eta}\frac{q_{m_{k}}^{4}}{\inf_{\theta}\ell_{\theta}}<T_{m_{k}}<\bigg(\frac{c\eta^{2}e^{\beta q_{m_{k}}}}{4C_{3}q_{m_{k}}^{6}}\bigg)^{\frac{1}{3}}.
\end{equation}
Indeed, the lower bound in \eqref{dyn:75} implies that the second term of $\varepsilon_{2}(T_{m_{k}})$ is less than the first. The upper bound in \eqref{dyn:75} is the second inequality in 
$$
\sup_{\theta\in\T}\varepsilon_{2}(T_{m_{k}})<2C_{3}e^{-\beta q_{m_{k}}}T_{m_{k}}^{2}<\frac{c\eta^{2}}{2q_{m_{k}}^{6}T_{m_{k}}}<\frac{1}{2}\inf_{\theta\in\T}\frac{c\eta^{2}}{q_{m_{k}}^{6}\ell_{\theta}T_{m_{k}}}.
$$ 

Now let us obtain the bounds \eqref{dyn:75}, on the subsequence $T_{m_{k}}$. For the upper bound we need to check that  $e^{\delta^{-1}(\gamma_{0}+\varepsilon')q_{m_{k}}}<\big(\frac{c\eta^{2}e^{\beta q_{m_{k}}}}{4C_{3}q_{m_{k}}^{6}}\big)^{\frac{1}{3}}$, which requires $\frac{\beta}{3}>\delta^{-1}(\gamma_{0}+\varepsilon')$ (and sufficiently large $k$). This holds automatically since we'd initially fixed $\frac{\beta}{3}=\delta^{-1}(\gamma_{0}+2\varepsilon')$.  For the lower bound we need to check that $\frac{2\beta}{c_{1}c_{3}\eta}\frac{q_{m_{k}}^{4}}{\inf_{\theta}\ell_{\theta}} < e^{\delta^{-1}(\gamma_{0}+\varepsilon')q_{m_{k}}}$. Indeed, \eqref{dyn:31} states that  $1/\inf_{\theta}\ell_{\theta}<e^{(\gamma_{0}+\varepsilon'')q_{m_{k}}}
$, so comparing the exponents, we require $\varepsilon''<\varepsilon'$ (which we have already fixed) and $\delta<1$, which follows automatic from our assumption that $\delta<\frac{1}{2}$, in the statement of the theorem.

For \eqref{dyn:30} to hold, we must also ensure that $C\eta^{-1}q_{m_{k}}^{4}\ell_{\theta}^{-1}<\frac{1}{2}c_{1}\eta q_{m_{k}}^{-4}\ell_{\theta} T_{m_{k}}$, which simply follows from  $\delta<\frac{1}{2}$ and $1/\inf_{\theta}\ell_{\theta}<e^{(\gamma_{0}+\varepsilon'')q_{m_{k}}}
$.
\end{proof}

\section{Proof of Lemma \ref{dyn:uniformweakintervals}}\label{dyn:secuniflowerbound}

Gordon's lemma (see e.g. \cite[Theorem 10.3]{CFKS09}) rules out $\ell^{2}$-solutions to the eigenvalue equation of $H_{\alpha,\theta}$ for any energy $E\in\sigma_{\alpha}$ in the case that the frequency $\alpha\in\R\setminus\Q$ is extremely Liouville, in the sense that there exists a sequence of rationals $\alpha_{m}=\frac{p_{m}}{q_{m}}$ for which $|\alpha-\alpha_{m}|<Ce^{-\beta q_{m}}$ where $\beta=\beta_{m}=\log(m)\rightarrow\infty$ (and hence $\beta(\alpha)=+\infty$). We provide the modification of their proof to show that $H_{\alpha,\theta}\psi=E\psi$ has no $\ell^{2}$-solution if  $2\gamma(E)<\beta$. 
 
\begin{lemma}\label{dyn:nonatomic}
Let $H_{\alpha,\theta}$ be a bounded discrete one-dimensional Schr\"odinger operator \eqref{dyn:01} with $\theta\in\mathbb{T}$, $\alpha\in\R\setminus\Q$ and associated continuous Lyapunov exponent $\gamma$. Let $\alpha_{m}=\frac{p_{m}}{q_{m}}\rightarrow\alpha$ be a sequence of rationals. Fix $\beta>0$ and assume $|\alpha-\alpha_{m}|<e^{-\beta q_{m}}$ for all $m\geq1$.  If $2\gamma(E)<\beta$ then $\sup_{\theta\in\T}\mu_{\alpha,\theta}(\{E\})=0$. 
\end{lemma}

\begin{proof}[Proof of Lemma \ref{dyn:uniformweakintervals}]
Let $C_{b}(\R)$ denote the set of bounded continuous functions $g:\R\rightarrow\C$. For simplicity of notation denote $B_{\varepsilon}=B_{\varepsilon}(E_{0})$. For any $l\geq1$, define the set $F_{l}=\{E\in\R:\text{dist}(E,B_{\varepsilon})\geq 1/l\}$ and 
$$
g_{l}(E)=\frac{\text{dist}(E,F_{l})}{\text{dist}(E,B_{\varepsilon}) + \text{dist}(E,F_{l})}
$$ which coincides with the characteristic function $\chi_{B_{\varepsilon}}$ on the set $B_{\varepsilon}\cup F_{l}$ and coincides with linear functions on both the left and the right interval of the set $B_{\varepsilon}^{\mathsf{c}}\cap F_{l}^{\mathsf{c}}$. 
Define the triangle functions 
$$
g_{l,\pm}(E)=(1-l|E-(E_{0}\pm\varepsilon)|)\chi_{B_{1/l}(E_{0}\pm\varepsilon)}(E)
$$ centred at either edge $E_{0}\pm\varepsilon$ of the ball. Clearly $g_{l},g_{l,\pm}\in C_{b}(\R)$. For any $g\in C_{b}(\R)$, we have
\begin{equation}\label{dyn:41}
\lim_{\theta'\rightarrow\theta}\int_{\R}g\,d\mu_{\alpha,\theta'}=\int_{\R}g\,d\mu_{\alpha,\theta}
\text{ and }
\lim_{m\rightarrow\infty}\sup_{\theta\in\T}\sup_{\varkappa\in[0,\frac{\pi}{q_{m}}]}\bigg|\int_{\R}g \,d\mu_{\alpha_{m},\theta}^{(\varkappa)}-\int_{\R}g \,d\mu_{\alpha,\theta}\bigg| = 0 
\end{equation}
the proof of weak convergence \eqref{dyn:41} is standard and is provided at the end of this section. 

It follows from $|\chi_{B_{\varepsilon}}-g_{l}|\leq g_{l,-}+g_{l,+}$, that 
\begin{equation}\label{dyn:42}
\big| \mu_{\alpha,\theta}(B_{\varepsilon})-  \mu_{\alpha,\theta'}(B_{\varepsilon})\big| \leq \int_{\R} g_{l,-}+g_{l,+} \,d(\mu_{\alpha,\theta}+\mu_{\alpha,\theta'})+\int_{\R} g_{l} \,d(\mu_{\alpha,\theta}-\mu_{\alpha,\theta'})
\end{equation} and weak convergence that 
$$
\limsup_{\theta'\rightarrow\theta} \big| \mu_{\alpha,\theta}(B_{\varepsilon})-  \mu_{\alpha,\theta'}(B_{\varepsilon})\big| \leq 2\int_{\R}g_{l,-}+g_{l,+} \,d\mu_{\alpha,\theta}
$$
and 
$$
\limsup_{m\rightarrow\infty}\sup_{\theta\in\T}\sup_{\varkappa\in[0,\frac{\pi}{q_{m}}]}\big| \mu_{\alpha_{m},\theta}^{(\varkappa)}(B_{\varepsilon})-  \mu_{\alpha,\theta}(B_{\varepsilon})\big| \leq 
2\sup_{\theta\in\T}\int_{\R}g_{l,+}+g_{l,-}\,d\mu_{\alpha,\theta}.
$$
So the claim follows from 
\begin{equation}\label{dyn:44}
\limsup_{l\rightarrow \infty}\sup_{\theta\in\T} \int_{\R}g_{l,\pm}\,d\mu_{\alpha,\theta}=0. 
\end{equation}

Let us show that Lemma \ref{dyn:nonatomic} implies \eqref{dyn:44}. Indeed, if not, then there exists $\theta\in\T$ and $l_{j}\rightarrow\infty$, $\theta_{j}\rightarrow\theta\in\T$ such that 
$$
\int_{\R}g_{l_{j},\pm}\,d\mu_{\alpha,\theta_{j}}>\delta>0
$$ 
for all $j\geq1$. Since $l\leq l_{j}$ implies $g_{l,\pm}\geq g_{l_{j},\pm}$, it follows that for any $l<\infty$ there exists $j_{0}=j_{0}(l)<\infty$ such that $l_{j}\geq l$ for all $j>j_{0}$ and hence  $ \int_{\R}g_{l,\pm}\,d\mu_{\alpha,\theta_{j}}>\delta$ for all $j>j_{0}$. Weak convergence implies 
$$
\delta < \lim_{j\rightarrow\infty}\int_{\R}g_{l,\pm}\,d\mu_{\alpha,\theta_{j}}=\int_{\R}g_{l,\pm}\,d\mu_{\alpha,\theta}
$$ 
for every $l<\infty$, yet since by continuity we have $2\gamma(E_{0}\pm\varepsilon)<\beta$, Lemma \ref{dyn:nonatomic} implies 
$$
\limsup_{l\rightarrow\infty}\int_{\R}g_{l,\pm}\,d\mu_{\alpha,\theta}\leq \limsup_{l\rightarrow\infty}\mu_{\alpha,\theta}(B_{1/l}(E_{0}\pm\varepsilon)) = \mu_{\alpha,\theta}(\{E_{0}\pm\varepsilon\})=0
$$
since $\mu_{\alpha,\theta}$ is finite. 
\end{proof}

\begin{proof}[Proof of Lemma \ref{dyn:nonatomic}]
For $\alpha\in\R$, $\theta\in\T$ and $n\in\Z$, denote 
$$T_{\alpha,\theta,E}(n)=\begin{pmatrix}
E-V_{\alpha,\theta}(n) & -1 \\
1 & 0
\end{pmatrix}.$$ The transfer matrices 
\begin{equation}\label{dyn:36}
\Phi_{\alpha,\theta,[0,n]}(E)=T_{\alpha,\theta,E}(n)\cdots T_{\alpha,\theta,E}(0),\quad \Phi_{\alpha,\theta,[-n,-1]}^{-}(E)=T_{\alpha,\theta,E}^{-1}(-n)\cdots T_{\alpha,\theta,E}^{-1}(-1)
\end{equation} 
satisfy 
$$
\Phi_{\alpha,\theta,[0,n]}(E) u_{0} = u_{n+1} \text{ and } \Phi_{\alpha,\theta,[-n,-1]}^{-}(E) u_{0} = u_{-n}, \quad u_{n}=\begin{pmatrix}
\psi(n) \\
\psi(n-1) 
\end{pmatrix}
$$
for every $n\geq0$ where $\psi$ solves $H_{\alpha,\theta}\psi=E\psi$ for some $E\in\R$.

By comparing the norm of the difference of the transfer matrices associated to the operators $H_{\alpha,\theta}$ and $H_{\alpha_{m},\theta}$, we obtain: For any unit vector $u\in\C^{2}$ and $\varepsilon>0$, 
\begin{equation}\label{dyn:37}
\max_{0\leq n\leq2q_{m}}\|(\Phi_{\alpha,\theta,[0,n]}(E)-\Phi_{\alpha_{m},\theta,[0,n]}(E))u\|<Cq_{m}^{2}e^{(-\beta+2\gamma(E)+\varepsilon)q_{m}}
\end{equation}
for sufficiently large $m$ depending on $E,\varepsilon,\theta$. Indeed, first note
$$
\|\Phi_{\alpha,\theta,[0,n]}-\Phi_{\alpha_{m},\theta,[0,n]}\|\leq n\|f'\|_{\infty}e^{-\beta q_{m}}\sum_{j=0}^{n}\|\Phi_{\alpha,\theta,[j+1,n]}\|\|\Phi_{\alpha_{m},\theta,[0,j-1]}\|
$$ 
then split the sum $\sum_{j=0}^{n}=\sum_{N_{0}<j<n-N_{0}}+\sum_{j\leq N_{0}}+\sum_{j\geq n-N_{0}}$. For the latter two sums we use $\|\Phi_{\alpha,\theta,[0,j]}(E)\|\leq (2+|E|+\|f\|_{\infty})^{j+1}$ and for the former, we use 
\begin{equation}\label{dyn:38}
\|\Phi_{\alpha,\theta,[0,n]}(E)\|,\|\Phi_{\alpha_{m},\theta,[0,n]}(E)\|<e^{n(\gamma(E)+\varepsilon')}
\end{equation} which holds for any $\varepsilon'>0$, $E\in\R$, $\theta\in\T$, provided $m>M_{0}(\varepsilon',E,\theta)$ and $n>N_{0}(\varepsilon',E,\theta)$. Furman \cite[Theorem 1]{Fu97} proved \eqref{dyn:38} for the quasiperiodic transfer matrices. By an approximation argument (see e.g.\ \cite{Ha22}) applied to Furman, one obtains \eqref{dyn:38} for the periodic transfer matrices.

The transfer matrices associated with the periodic operator $H_{\alpha_{m},\theta}$ exhibit the relations  
$$
\Phi^{-}_{[-q_{m},-1]}=(\Phi_{[0,q_{m}-1]})^{-1},\quad (\Phi_{[0,q_{m}-1]})^{2}=\Phi_{[0,2q_{m}-1]},\quad (\Phi^{-}_{[-q_{m},-1]})^{2}=\Phi^{-}_{[-2q_{m},-1]}.
$$
which we combine with the following elementary fact (see e.g.\ \cite[Lemma 7.6]{Si82}): If $A$ is an invertible $2\times2$ matrix and the vector $u\in\C^{2}$ has unit norm $\|u\|=1$, then 
$$
\max(\|A^{2}u\|,\|Au\|,\|A^{-1}u\|,\|A^{-2}u\|)\geq\frac{1}{2}. 
$$
Setting $A=\Phi_{\alpha_{m},\theta,[0,q_{m}-1]}(E)$, \eqref{dyn:37} implies, for any unit vector $u\in\C^{2}$,
\begin{equation*}
\begin{split}
&\limsup_{m\rightarrow\infty}\max(\|\Phi_{\alpha,\theta,[0,q_{m-1}]}(E)u\|,\|\Phi_{\alpha,\theta,[0,2q_{m}-1]}(E)u\|,\|\Phi^{-}_{\alpha,\theta,[-q_{m},-1]}(E)u\|,\|\Phi^{-}_{\alpha,\theta,[-2q_{m},-1]}(E)u\|) \\
&\geq \limsup_{m\rightarrow\infty}(\max(\|A^{2}u\|,\|Au\|,\|A^{-1}u\|,\|A^{-2}u\|)-Cq_{m}^{2}e^{(-\beta+2\gamma(E)+\varepsilon)q_{m}}) \geq \frac{1}{2}
\end{split}
\end{equation*}
which implies that the eigenvalue equation $H_{\alpha,\theta}\psi=E\psi$, has no square summable solution $\psi\in\ell^{2}(\Z)$. Since $2\gamma(E)<\beta$, it follows that \eqref{dyn:37} tends to zero for sufficiently small $\varepsilon>0$.

If $\mu_{\alpha,\theta}(\{E\})=\|\chi_{\{E\}}(H_{\alpha,\theta})\delta_{0}\|^{2}+\|\chi_{\{E\}}(H_{\alpha,\theta})\delta_{1}\|^{2}>0$, then $\chi_{\{E\}}(H_{\alpha,\theta})\neq0$ so there exists a non trivial $\psi\in\text{Ran}(\chi_{\{E\}}(H_{\alpha,\theta}))\subset\ell^{2}(\Z)$ for which the Borel functional calculus implies   
$H_{\alpha,\theta}\psi=H_{\alpha,\theta}\chi_{\{E\}}(H_{\alpha,\theta})\psi=E\chi_{\{E\}}(H_{\alpha,\theta})\psi=E\psi$. But no such $\psi\in\ell^{2}(\Z)$ exists. So $\mu_{\alpha,\theta}(\{E\})=0$. The same is true for any phase $\theta\in\T$. 
\end{proof}

\emph{Weak convergence, ballistic bound and \eqref{dyn:27}:} All of which follow from the  Combes -- Thomas estimate (see e.g.\ \cite[Theorem 11.2]{Ki07}); there exists  $c>0$ such that for any bounded Schr\"odinger operator $H:\ell^{2}(\Z)\rightarrow\ell^{2}(\Z)$, 
\begin{equation}\label{dyn:33}
|\langle\delta_{n},(H-z)^{-1}\delta_{m}\rangle|\leq \frac{2}{\text{dist}(z,\sigma(H))}e^{-c\min(\text{dist}(z,\sigma(H)),1)|n-m|}
\end{equation} for any $n,m\in\Z$ and $z\in\C\setminus\sigma(H)$, which also holds true for the restrictions of the operator $H$ to a finite interval with Dirichlet boundary conditions. This version \eqref{dyn:33} of Combes -- Thomas is not written in the most general or optimal way. The constant $c>0$ is universal in the sense that it does not depend on the potential. 

Let us briefly comment on the weak convergence of the spectral measures stated in \eqref{dyn:41}. For the first statement of \eqref{dyn:41} we need to check that the spectral measure of the infinite volume operator $H_{\alpha,\theta'}$ (with phase $\theta'\in\T$) converges weakly to the the spectral measure of the infinite volume operator $H_{\alpha,\theta}$ (with phase $\theta\in\T$) as $\theta'\rightarrow\theta$.  It is a standard fact from the theory of weak convergence of measures that weak convergence is equivalent to the pointwise convergence of the associated characteristic functions. Namely, it suffice to show that  $\langle \delta_{0},e^{it H_{\alpha,\theta'}}\delta_{0}\rangle\rightarrow\langle\delta_{0},e^{it H_{\alpha,\theta}}\delta_{0}\rangle$ as $\theta'\rightarrow\theta$ for each $t\in\R$. For similar reasons to \eqref{dyn:66}, it suffice to show convergence of the $(0,0)$-entry of the resolvents,  via the second resolvent identity and the Combes -- Thomas estimate \eqref{dyn:33}.

For the second limit \eqref{dyn:41} in which we require the weak convergence of the spectral measure of the Floquet matrix uniformly in $\theta\in\T$ and $\varkappa\in[0,\frac{\pi}{q_{m}}]$. By the Stone-Weierstrass theorem, it is enough to check that the uniform limit holds for the function $g_{z}(E)=(E-z)^{-1}$ for any fixed $z\in\C$ outside of the real line $\Im(z)\neq 0$. Take $\Im(z)\neq 0$ and denote 
$$
R_{\alpha,\theta, z} = (H_{\alpha,\theta}-z)^{-1} \text{ and } R_{\alpha_{m},\theta, z}^{(\varkappa)} = (A_{\alpha_{m},\theta}(\varkappa)-z)^{-1}.
$$ 
The second resolvent identity gives: 
$$
|(R_{\alpha_{m},\theta,z}^{(\varkappa)}-R_{\alpha,\theta,z})(0,0)| \leq \sum_{n\in\Z; |j|\leq \frac{q_{m}}{2}} |R_{\alpha_{m},\theta,z}^{(\varkappa)}(0,j)||(H_{\alpha,\theta}-A_{\alpha_{m},\theta}(\varkappa))(j,n)||R_{\alpha,\theta,z}(n,0)|
$$
which is bounded by $C_{z}q_{m}^{2}(e^{-c_{z}\frac{q_{m}}{2}}+e^{-\beta q_{m}})$. Indeed, first use $|R_{\alpha_{m},\theta,z}^{(\varkappa)}(0,j)|\leq \frac{1}{|\Im(z)|}$, then split $\sum_{n\in\Z;|j|\leq \frac{q_{m}}{2}} = \sum_{|n|\geq \frac{q_{m}}{2}-1;|j|\leq \frac{q_{m}}{2}}+\sum_{|n|<\frac{q_{m}}{2}-1;|j|\leq \frac{q_{m}}{2}}$ and apply the Combes -- Thomas estimate to the term $|R_{\alpha,\theta,z}(n,0)|$, and for the second sum note that $|(H_{\alpha,\theta}-A_{\alpha_{m},\theta}(\varkappa))(j,n)|$ is the difference between the two potentials.

The ballistic bound follows from the Combes -- Thomas estimate \eqref{dyn:33} and ensures that the moments \eqref{dyn:02} exist. Indeed, by applying the spectral theorem \eqref{dyn:04} followed by the Cauchy integral formula, we get
\begin{equation}\label{dyn:65}
\langle\delta_{n},e^{-itH}\delta_{0}\rangle = -\frac{1}{2\pi i}\oint_{\mathcal{C}}e^{-itz}\langle\delta_{n},(H-z)^{-1}\delta_{0}\rangle\,dz
\end{equation} where the contour $\mathcal{C}$ encircles the spectrum counterclockwise. To obtain the ballistic bound, let us take the contour $\mathcal{C}$ to be the boundary of the rectangle with $|\Im(z)|\leq1$ and $|\Re(z)|\leq \|H\|+1$. The Combes -- Thomas implies $|\langle\delta_{n},(H-z)^{-1}\delta_{0}\rangle|\leq 2e^{-c|n|}$ for any $z\in\mathcal{C}$. The formula \eqref{dyn:65} then gives $|\langle\delta_{n},e^{-itH}\delta_{0}\rangle|\leq e^{t-c|n|}\frac{1}{\pi}\oint_{\mathcal{C}}|dz|$
which implies the ballistic bound
\begin{equation}\label{dyn:ballistic}
|\langle\delta_{n},e^{-itH}\delta_{0}\rangle|\leq Ce^{-\frac{1}{2}c|n|}, \quad \forall |n|>2c^{-1}t
\end{equation}
 where $\pi C=4(\|H\|+2)$ is the circumference of the rectangle $\mathcal{C}$. The ballistic bound \eqref{dyn:ballistic} also holds for the restriction of $H$ to a finite interval with Dirichlet boundary conditions.

\begin{proof}[Proof of \eqref{dyn:27}]
Let $\mathcal{C}$ denote the same rectangle as above.  Let $H_{N}$ denote the matrix given by the restriction of $H$ to the finite interval $[-N,N]\subset\Z$ for $N\geq0$. To obtain \eqref{dyn:27}, we split the problem into two separate cases, $|n|>2c^{-1}t$ and $|n|\leq 2c^{-1}t$.

The first case follows from the ballistic bound. Indeed, the ballistic bound \eqref{dyn:ballistic} also holds for the matrix $H_{N}$: $|\langle\delta_{n},e^{-itH_{N}}\delta_{0}\rangle|\leq Ce^{-\frac{1}{2}c|n|} $, for every $|n|>2c^{-1}t$ and $N\geq0$, with the same constants $c>0$ and $\pi C=4(\|H\|+2)$, since we have $\|H_{N}\|\leq\|H\|$ for every $N\geq0$. The triangle inequality then gives 
\begin{equation}\label{dyn:34}
|\langle\delta_{n},e^{-itH}\delta_{0}\rangle-\langle\delta_{n},e^{-itH_{N}}\delta_{0}\rangle|\leq 2Ce^{-\frac{1}{2}c|n|}=2Ce^{-\max(\frac{1}{2}c|n|,t)}
\end{equation} for every $N\geq0$. 

In the second case we use \eqref{dyn:65} again to deduce 
\begin{equation}\label{dyn:66}
|\langle\delta_{n},e^{-itH}\delta_{0}\rangle-\langle\delta_{n},e^{-itH_{N}}\delta_{0}\rangle|\leq \frac{Ce^{t}}{2}\max_{z\in\mathcal{C}}|\langle\delta_{n},(H-z)^{-1}\delta_{0}\rangle-\langle\delta_{n},(H_{N}-z)^{-1}\delta_{0}\rangle|
\end{equation}
then the second resolvent identity and the Combes -- Thomas estimate show that the maximum is bounded by $C_{1}e^{t-cN}= C_{1}e^{(1-cC')t}$, since $N=C't$. For sufficiently large $C'$, 
$$
|\langle\delta_{n},e^{-itH}\delta_{0}\rangle-\langle\delta_{n},e^{-itH_{N}}\delta_{0}\rangle|\leq C_{2}e^{-c_{1}t} = C_{2}e^{-c_{1}\max(\frac{1}{2}c|n|,t)}
$$ which, combined with \eqref{dyn:34}, implies \eqref{dyn:27}.
\end{proof}

\itshape{Address:} \scshape{School of Mathematical Sciences, Queen Mary University of London, London E1 4NS, United Kingdom.} 

\itshape{E-mail:} \scshape{l.haeming@qmul.ac.uk.}


\begin{thebibliography}{1}

\bibitem[AvSi82]{AvSi82}
Avron, J. and Simon, B., 1982. \emph{Singular continuous spectrum for a class of almost periodic Jacobi matrices.}

\bibitem[AvYoZh15]{AvYoZh15}
Avila, A., You, J. and Zhou, Q., 2017. \emph{Sharp phase transitions for the almost Mathieu operator.}

\bibitem[BoJi02]{BoJi02}
Bourgain, J. and Jitomirskaya, S., 2002. \emph{Continuity of the Lyapunov exponent for quasiperiodic operators with analytic potential.} Journal of statistical physics, 108(5), pp.1203-1218.

\bibitem[CFKS09]{CFKS09}
Cycon, H.L., Froese, R.G., Kirsch, W. and Simon, B., 2009. \emph{Schrödinger operators: With application to quantum mechanics and global geometry.} Springer.

\bibitem[dRJiLaSi96]{dRJiLaSi96}
del Rio, R., Jitomirskaya, S., Last, Y. and Simon, B., 1996. \emph{Operators with singular continuous spectrum, IV. Hausdorff dimensions, rank one perturbations, and localization.} Journal d'Analyse Mathématique, 69, pp.153-200.

\bibitem[DaTc07]{DaTc07}
Damanik, D. and Tcheremchantsev, S., 2007. \emph{Upper bounds in quantum dynamics.} Journal of the American Mathematical Society, 20(3), pp.799-827.

\bibitem[Fu97]{Fu97}
Furman, A., 1997, January. \emph{On the multiplicative ergodic theorem for uniquely ergodic systems.} In Annales de l'Institut Henri Poincare (B) Probability and Statistics (Vol. 33, No. 6, pp. 797-815). No longer published by Elsevier.
Vancouver	

\bibitem[Go76]{Go76}
Gordon, A.Y., 1976. \emph{The point spectrum of the one-dimensional Schrodinger operator.} Uspekhi Matematicheskikh Nauk, 31(4), pp.257-258.


\bibitem[GuSB02]{GuSB02}
Guarneri, I. and Schulz-Baldes, H., 2002. \emph{Lower bounds on wave packet propagation by packing dimensions of spectral measures.} In Mathematical Physics Electronic Journal: (Print Version) Volumes 5 and 6 (pp. 1-16).

\bibitem[HaJi18]{HaJi18}
Han, R. and Jitomirskaya, S., 2018. \emph{Quantum dynamical bounds for ergodic potentials with underlying dynamics of zero topological entropy.} Analysis \& PDE, 12(4), pp.867-902.

\bibitem[Ha22]{Ha22}
Haeming, L., 2022. \emph{On the Bandwidths of Periodic Approximations to Discrete Schr\" odinger Operators.} arXiv preprint arXiv:2208.01646. To appear in Journal d'Analyse Mathématique. 


\bibitem[Ji94]{Ji94}
Jitomirskaya, S., 1994. \emph{Almost everything about the almost Mathieu operator.} II. In XIth International Congress of Mathematical Physics (Paris, 1994) (pp. 373-382).

\bibitem[JiMa16]{JiMa16}
Jitomirskaya, S. and Mavi, R., 2016. \emph{Dynamical bounds for quasiperiodic Schrödinger operators with rough potentials.} International Mathematics Research Notices, 2017(1), pp.96-120.

\bibitem[JiYa17]{JiYa17}
Jitomirskaya, S. and Yang, F., 2017. \emph{Singular continuous spectrum for singular potentials.} Communications in mathematical physics, 351, pp.1127-1135.

\bibitem[JiLi17]{JiLi17}
Jitomirskaya, S. and Liu, W., 2017. \emph{Arithmetic spectral transitions for the Maryland model.} Communications on Pure and Applied Mathematics, 70(6), pp.1025-1051.

\bibitem[JiZh21]{JiZh21}
Jitomirskaya, S. and Zhang, S., 2021. \emph{Quantitative continuity of singular continuous spectral measures and arithmetic criteria for quasiperiodic Schrödinger operators.} Journal of the European Mathematical Society, 24(5), pp.1723-1767.

\bibitem[JiLi21]{JiLi21}
Jitomirskaya, S. and Liu, W., 2021. \emph{Upper bounds on transport exponents for long-range operators.} Journal of Mathematical Physics, 62(7).

\bibitem[JiPo22]{JiPo22}
Jitomirskaya, S. and Powell, M., 2022. \emph{Logarithmic quantum dynamical bounds for arithmetically defined ergodic Schrödinger operators with smooth potentials.} In Analysis at Large: Dedicated to the Life and Work of Jean Bourgain (pp. 173-201). Cham: Springer International Publishing.

\bibitem[Ko84]{Ko84}
Kotani, S., 1984. \emph{Ljapunov indices determine absolutely continuous spectra of stationary random one-dimensional Schrödinger operators.} In North-Holland Mathematical Library (Vol. 32, pp. 225-247). Elsevier.

\bibitem[Ki07]{Ki07}
Kirsch, W., 2007. \emph{An invitation to random Schrödinger operators.} arXiv preprint arXiv:0709.3707.

\bibitem[La94]{La94}
Last, Y., 1994. \emph{Zero measure spectrum for the almost Mathieu operator.} Communications in mathematical physics, 164(2), pp.421-432.

\bibitem[La96]{La96}
Last, Y., 1996. \emph{Quantum dynamics and decompositions of singular continuous spectra.} Journal of Functional Analysis, 142(2), pp.406-445.

\bibitem[Li23]{Li23}
Liu, W., 2023. \emph{Power law logarithmic bounds of moments for long range operators in arbitrary dimension.} Journal of Mathematical Physics, 64(3).

\bibitem[Si82]{Si82}
Simon, B., 1982. \emph{Almost periodic Schrödinger operators: a review.} Advances in Applied Mathematics, 3(4), pp.463-490.

\bibitem[Si07]{Si07}
Simon, B., 2007. \emph{Equilibrium measures and capacities in spectral theory.} arXiv preprint arXiv:0711.2700.

\bibitem[ShSo23]{ShSo23}
Shamis, M. and Sodin, S., 2023. \emph{Upper bounds on quantum dynamics in arbitrary dimension.} Journal of Functional Analysis, 285(7), p.110034.

\bibitem[Vi14]{Vi14}
Viana, M., 2014. \emph{Lectures on Lyapunov exponents} (Vol. 145). Cambridge University Press.

\bibitem[YaZh19]{YaZh19}
Yang, F. and Zhang, S., 2019, July. \emph{Singular continuous spectrum and generic full spectral/packing dimension for unbounded quasiperiodic Schrödinger operators.} In Annales Henri Poincaré (Vol. 20, pp. 2481-2494). Springer International Publishing.


\end{thebibliography}
\end{document}